\documentclass[a4paper,11pt,fleqn]{article}
\usepackage{amsmath}
\usepackage{amssymb}
\usepackage{theorem}
\usepackage{mathrsfs} 
\usepackage[usenames,dvipsnames]{pstricks}
\usepackage{pstricks-add}
\usepackage{pst-plot}
\usepackage{pst-math}
\usepackage{euscript}
\usepackage{graphicx}
\usepackage{enumitem}
\usepackage{charter}
\usepackage{empheq}

\newcommand{\email}[1]{\href{mailto:#1}{\nolinkurl{#1}}}
\definecolor{labelkey}{rgb}{0,0.08,0.45}
\definecolor{refkey}{rgb}{0,0.6,0.0}
\definecolor{Brown}{rgb}{0.45,0.0,0.05}
\definecolor{dgreen}{rgb}{0.00,0.59,0.00}
\definecolor{dblue}{rgb}{0,0.08,0.75}
\RequirePackage[dvips,colorlinks,hyperindex]{hyperref} 
\hypersetup{linktocpage=true,citecolor=dblue,linkcolor=dgreen}
\PassOptionsToPackage{normalem}{ulem}
\usepackage{ulem}

\renewcommand{\leq}{\ensuremath{\leqslant}}
\renewcommand{\geq}{\ensuremath{\geqslant}}

\newcommand{\Frac}[2]{\displaystyle{\frac{#1}{#2}}} 
 
\newcommand{\scal}[2]{{\left\langle{{#1}\mid{#2}}\right\rangle}}

\newcommand{\menge}[2]{\big\{{#1}~\big |~{#2}\big\}} 
 
\newcommand{\HHH}{{\ensuremath{\boldsymbol{\mathcal H}}}}
\newcommand{\XXX}{{\ensuremath{\boldsymbol{\mathcal X}}}}
\newcommand{\YYY}{{\ensuremath{\boldsymbol{\mathcal Y}}}}

\newcommand{\HH}{\ensuremath{{\mathcal H}}}
\newcommand{\moyo}[2]{\ensuremath{\sideset{^{#2}}{}%
{\operatorname{}}\!\!#1}}
\newcommand{\AL}{\ensuremath{\EuScript A}}
\newcommand{\GG}{\ensuremath{{\mathcal G}}}

\newcommand{\XX}{\ensuremath{\mathcal{X}}}
\newcommand{\YY}{\ensuremath{\mathcal{Y}}}

\newcommand{\Sum}{\ensuremath{\displaystyle\sum}}
\newcommand{\Prod}{\ensuremath{\displaystyle\prod}}
\newcommand{\emp}{\ensuremath{{\varnothing}}}

\newcommand{\Id}{\ensuremath{\operatorname{Id}}}
\newcommand{\ID}{{\boldsymbol{\Id}\,}}
\newcommand{\cart}{\ensuremath{\raisebox{-0.5mm}{\mbox{\LARGE{$\times$}}}}}

\newcommand{\RR}{\ensuremath{\mathbb{R}}}
\newcommand{\RP}{\ensuremath{\left[0,+\infty\right[}}
\newcommand{\RM}{\ensuremath{\left]-\infty,0\right]}}

\newcommand{\BL}{\ensuremath{\EuScript B}\,}
\newcommand{\RPP}{\ensuremath{\left]0,+\infty\right[}}

\newcommand{\RX}{\ensuremath{\left]-\infty,+\infty\right]}}

\newcommand{\NN}{\ensuremath{\mathbb N}}

\newcommand{\bK}{\ensuremath{\mathbb K}}
\newcommand{\intdom}{\ensuremath{\text{int\,dom}\,}}
\newcommand{\weakly}{\ensuremath{\:\rightharpoonup\:}}
\newcommand{\exi}{\ensuremath{\exists\,}}
\newcommand{\ran}{\ensuremath{\text{\rm ran}\,}}
\newcommand{\cran}{\ensuremath{\overline{\text{\rm ran}}\,}}
\newcommand{\zer}{\ensuremath{\text{\rm zer}\,}}
\newcommand{\pinf}{\ensuremath{{+\infty}}}
\newcommand{\minf}{\ensuremath{{-\infty}}}
\newcommand{\dom}{\ensuremath{\text{\rm dom}\,}}

\newcommand{\prox}{\ensuremath{\text{\rm prox}}}
\newcommand{\Fix}{\ensuremath{\text{\rm Fix}\,}}

\newcommand{\sign}{\ensuremath{\text{\rm sign}}}

\newcommand{\conv}{\ensuremath{\text{\rm conv}\,}}

\newcommand{\zeroun}{\ensuremath{\left]0,1\right[}}   
\newcommand{\rzeroun}{\ensuremath{\left]0,1\right]}}   
   
\newtheorem{theorem}{Theorem}[section]
\newtheorem{lemma}[theorem]{Lemma}

\newtheorem{proposition}[theorem]{Proposition}
\theoremstyle{plain}{\theorembodyfont{\rmfamily}%
\newtheorem{model}[theorem]{Model}}
\theoremstyle{plain}{\theorembodyfont{\rmfamily}%
\newtheorem{notation}[theorem]{Notation}}
\theoremstyle{plain}{\theorembodyfont{\rmfamily}%
\newtheorem{condition}[theorem]{Condition}}
\theoremstyle{plain}{\theorembodyfont{\rmfamily}%
\newtheorem{assumption}[theorem]{Assumption}}
\theoremstyle{plain}{\theorembodyfont{\rmfamily}%
}
\theoremstyle{plain}{\theorembodyfont{\rmfamily}%
}
\theoremstyle{plain}{\theorembodyfont{\rmfamily}%
\newtheorem{example}[theorem]{Example}}
\theoremstyle{plain}{\theorembodyfont{\rmfamily}%
\newtheorem{remark}[theorem]{Remark}}
\theoremstyle{plain}{\theorembodyfont{\rmfamily}%
\newtheorem{definition}[theorem]{Definition}}
\theoremstyle{plain}{\theorembodyfont{\rmfamily}%
}

\numberwithin{equation}{section}
\begin{document}
\title{\sffamily \vskip -9mm 
Deep Neural Network Structures Solving Variational 
Inequalities\thanks{Contact author: 
P. L. Combettes, \email{plc@math.ncsu.edu}, phone: +1 919 515-2671. 
The work of P. L. Combettes was supported by the 
National Science Foundation under grant CCF-1715671. 
The work of J.-C. Pesquet was supported by Institut 
Universitaire de France.}}
\author{Patrick L. Combettes$^1$ and 
Jean-Christophe Pesquet$^2$
\\[4mm]
\small
\small $\!^1$North Carolina State University,
Department of Mathematics, Raleigh, NC 27695-8205, USA\\
\small\email{plc@math.ncsu.edu}\\[2mm]
\small $\!^2$CentraleSup\'elec, Inria, Universit\'e Paris-Saclay,
Center for Visual Computing, 91190 Gif sur Yvette, France\\
\small\email{jean-christophe@pesquet.eu}
}
\date{~}
\maketitle
\thispagestyle{empty}

\vskip -10mm

\noindent
{\bfseries Abstract.}
Motivated by structures that appear in deep neural networks, we
investigate nonlinear composite models alternating proximity and
affine operators defined on different spaces. We first show that a
wide range of activation operators used in neural networks are
actually proximity operators. We then establish conditions for the
averagedness of the proposed composite constructs and investigate
their asymptotic properties. It is shown that the limit of the
resulting process solves a variational inequality which, in
general, does not derive from a minimization problem.

\section{Introduction}

A powerful tool from fixed point theory to analyze and
solve optimization and inclusion problems in a real Hilbert space
$\HH$ is the class of averaged nonexpansive operators, which was 
introduced in \cite{Bail78}. Recall that an operator
$T\colon\HH\to\HH$ is \emph{nonexpansive} if it is $1$-Lipschitzian,
and $\alpha$-\emph{averaged} for some $\alpha\in\rzeroun$ if there 
exists a nonexpansive operator $Q\colon\HH\to\HH$ such that
$T=(1-\alpha)\Id+\alpha Q$; if $\alpha=1/2$, $T$ is 
\emph{firmly nonexpansive}. The importance of firmly nonexpansive
operators in convex optimization and variational methods has long
been recognized \cite{Cras95,Ecks92,Mart72,Roc76a,Tsen92}. 
More generally, averaged operators were shown in \cite{Baus96} to 
play a prominent role in the analysis of convex feasibility
problems. In this context the underlying problem is to find a
common fixed point of averaged operators. In \cite{Opti04}, it was
shown that many convex minimization and monotone
inclusion problems reduce to the more general problem of
finding a fixed point of
compositions of averaged operators, which provided a unified
analysis of various proximal splitting algorithms. Along these
lines, several fixed point methods based on various combinations
of averaged operators have since been devised, see 
\cite{Arag18,Atto18,Barg18,Baus15,Borw17,Botr17,Brav18,%
Cegi12,Cens16,Jmaa15,Cond13,Mour18,Yama17} for recent work. 
Motivated by deep neural network structures with thus far
elusive asymptotic properties, we investigate in the present paper
a novel averaged operator model involving a mix of nonlinear and
linear operators.

Artificial neural networks have attracted considerable attention 
as a tool to better understand, model, and imitate the human brain 
\cite{Hayk98,Mccu43,Rose58}. In a Hilbertian setting 
\cite{Barro93}, an $(n+1)$-layer feed-forward neural network 
architecture acting on real Hilbert spaces 
$(\HH_i)_{0\leq i\leq n}$ is defined as the composition of 
operators $R_n\circ (W_n\cdot+b_n)\circ\cdots\circ R_1\circ 
(W_1\cdot+b_1)$ where, for every $i\in\{1,\ldots,n\}$, 
$R_i\colon\HH_i\to\HH_i$ is a nonlinear operator known as
an activation operator, 
$W_i\colon\HH_{i-1}\to\HH_i$ is a linear operator, known as a
weight operator, and $b_i\in\HH_i$ is a so-called bias parameter. 
Deep neural networks feature a (possibly large) number $n$
of layers. In recent years, they have been found to be quite
successful in a wide array of classification, recognition, and
prediction tasks; see \cite{Lecu15} and its bibliography.
Despite their success, the operational structure and properties
of deep neural networks are not yet well
understood from a mathematical viewpoint. In the present paper, we
propose to analyze them within the following iterative model.
We emphasize that our purpose is not to study the training
of the network, which consists of optimally setting
the weight operators and bias parameters from
data samples, but to analyze mathematically such a structure 
once it is trained. Our model is also of general interest in
constructive fixed point theory.

\begin{model}
\label{m:2}
Let $m\geq 1$ be an integer, let $\HH$ and 
$(\HH_i)_{0\leq i\leq m}$ be nonzero real Hilbert spaces, 
such that $\HH_m=\HH_0=\HH$.
For every $i\in\{1,\ldots,m\}$ and every $n\in\NN$, 
let $W_{i,n}\colon\HH_{i-1}\to\HH_i$ be a bounded linear 
operator, let $b_{i,n}\in\HH_i$, and let 
$R_{i,n}\colon\HH_i\to\HH_i$. Let $x_0\in\HH$, let
$(\lambda_n)_{n\in\NN}$ be a sequence in $\RPP$, set
\begin{equation}
\label{e:sj4}
(\forall n\in\NN)(\forall i\in\{1,\ldots,m\})\quad 
T_{i,n}\colon\HH_{i-1}\to\HH_i\colon x\mapsto R_{i,n}(W_{i,n}
x+b_{i,n}),
\end{equation}
and iterate
\begin{equation}
\label{e:algo4}
\begin{array}{l}
\text{for}\;n=0,1,\ldots\\
\left\lfloor
\begin{array}{ll}
x_{1,n}&\!\!\!=T_{1,n} x_n\\
x_{2,n}&\!\!\!=T_{2,n} x_{1,n}\\
       &\hskip -1mm\vdots\\
x_{m,n}&\!\!\!=T_{m,n} x_{m-1,n}\\
x_{n+1}&\!\!\!=x_n+\lambda_n(x_{m,n}-x_n).
\end{array}
\right.\\[2mm]
\end{array}
\end{equation}
\end{model}

In sharp contrast with existing algorithmic frameworks involving 
averaged operators (see cited works above), the operators
involved in Model~\ref{m:2} are not necessarily all defined on the
same Hilbert space and, in addition, they need not all be averaged.
Let us also note that the relaxation parameters
$(\lambda_n)_{n\in\NN}$ in \eqref{e:algo4}
allow us to model skip connections \cite{Svriv15}, in the spirit of
residual networks \cite{He16}. If $\lambda_n\equiv1$, we 
obtain the standard feed-forward architecture \cite{Hayk98}.

Our contributions are articulated around the following findings.
\begin{itemize}
\item
We show that a wide range of activation operators used in neural
networks are actually proximity operators, which paves the way to
the analysis of such networks via fixed point theory. 
\item
We provide a new analysis of compositions of proximity and affine
operators, establishing mild conditions that guarantee that the
resulting operator is averaged. 
\item
We show that, under suitable assumptions, the asymptotic output 
of the network converges to a point defined via a variational 
inequality. Furthermore, in general, this variational
inequality does not derive from a minimization problem.  
\end{itemize}

The remainder of the paper is organized as follows. In 
Section~\ref{sec:3}, we bring to light strong connections between 
the activation functions employed in neural networks and the 
theory of proximity operators in convex analysis. In 
Section~\ref{sec:2} we derive new results on the averagedness 
properties of compositions of proximity and 
affine operators acting on different spaces. 
In Section~\ref{sec:4}, we investigate the asymptotic behavior
of a class of deep neural networks structures and show that their
fixed points solve a variational inequality.
The main assumption on this subclass of Model~\ref{m:2} is that the
structure of the network is periodic in the sense that a group of
layers is repeated. Finally, in Section~\ref{sec:5}, the
same properties are established for a broader class of networks.

\noindent
{\bfseries Notation.} We follow standard notation from convex
analysis and operator theory \cite{Livre1,Rock70}. Thus, $\weakly$ 
and $\to$ denote, respectively, weak and strong convergence in 
$\HH$ and $\Gamma_0(\HH)$ is the class of lower semicontinuous 
convex functions $\varphi\colon\HH\to\RX$ such that 
$\dom\varphi=\menge{x\in\HH}{\varphi(x)<\pinf}\neq\emp$. Now let
$\varphi\in\Gamma_0(\HH)$. The conjugate of $\varphi$ is denoted by
$\varphi^*$, its subdifferential by $\partial\varphi$, and its 
proximity operator is $\prox_\varphi\colon\HH\to\HH\colon x\mapsto
\text{argmin}_{y\in\HH}(\varphi(y)+\|x-y\|^2/2)$. The symbols
$\ran T$, $\dom T$, $\Fix T$, and $\zer T$ denote respectively 
the range, the domain, the fixed point set, and the set of zeros
of an operator $T$. The space of bounded linear operators from a
Banach space $\XX$ to a Banach space $\YY$ is denoted by
$\BL(\XX,\YY)$. Finally, $\ell_+^1$ denotes the space of 
summable sequences in $\RP$.

\section{Proximal activation in neural networks}
\label{sec:3}

The following facts will be needed.

\begin{lemma}
\label{l:1}
Let $\varphi\in\Gamma_0(\HH)$. Then the following hold:
\begin{enumerate}
\itemsep0mm 
\item
\label{l:1i}
{\rm\cite[Proposition~12.29]{Livre1}}
$\Fix\prox_\varphi=\text{\rm Argmin}\,\varphi$.
\item
\label{l:1ii}
{\rm\cite[Corollary~24.5]{Livre1}}
Let $g\in\Gamma_0(\HH)$ be such that $\varphi=g-\|\cdot\|^2/2$. 
Then $\prox_\varphi=\nabla g^*$.
\end{enumerate}
\end{lemma}

\subsection{Activation functions}
\label{sec:31}

An activation function is a function $\varrho\colon\RR\to\RR$ which 
models the firing activity of neurons. The simplest instance, that
goes back to the perceptron machine \cite{Rose58}, is that of a
binary firing model: the neuron is either firing or at rest. For
instance, if the firing level is $1$ and the rest state is $0$, we
obtain the binary step function 
\begin{equation}
\varrho\colon\xi\mapsto
\begin{cases}
1,&\text{if}\;\;\xi>0;\\
0,&\text{if}\;\;\xi\leq 0,
\end{cases}
\end{equation}
which was initially proposed in \cite{Mccu43}.
As this discontinuous activation model may lead to unstable neural
networks, various continuous approximations have
been proposed. Our key observation is that a vast
array of activation functions used in neural networks 
belong to the following class. 

\begin{definition}
\label{d:1}
The set of functions from $\RR$ to $\RR$ which are increasing,
1-Lipschitzian, and take value $0$ at $0$ is denoted by
$\AL(\RR)$.
\end{definition}

Remarkably, we can precisely characterize this class of activation 
functions as that of proximity operators. 
 
\begin{proposition}
\label{p:1}
Let $\varrho\colon\RR\to\RR$. Then $\varrho\in\AL(\RR)$ if and 
only if there exists a function $\phi\in\Gamma_0(\RR)$, which 
has $0$ as a minimizer, such that $\varrho=\prox_\phi$.
\end{proposition}
\begin{proof}
The fact that the class of increasing, $1$-Lipschitzian 
functions from $\RR$ to $\RR$ coincides with that of proximity
operators of functions in $\Gamma_0(\RR)$ is shown in 
\cite[Proposition~2.4]{Siop07}. In view of 
Lemma~\ref{l:1}\ref{l:1i} and Definition~\ref{d:1}, the proof is
complete.
\end{proof}

To illustrate the above results, let us provide examples of common
activation functions $\varrho\in\AL(\RR)$, and identify 
the potential $\phi$ they derive from in Proposition~\ref{p:1}
(see Fig.~\ref{fig:2}).

\begin{example}
\label{ex:000}
The most basic activation function is $\varrho=\Id=\prox_0$. 
It is in particular useful in dictionary learning approaches, which 
correspond to the linear special case of 
Model~\ref{m:2} \cite{Tariy16}.
\end{example}

\begin{example}
\label{ex:0}
The saturated linear activation function \cite{Hayk98} 
\begin{equation}
\label{e:ex0}
\varrho\colon\RR\to\RR\colon\xi\mapsto
\begin{cases}
1,&\text{if}\;\;\xi>1;\\
\xi,&\text{if}\;\;-1\leq\xi\leq 1;\\
-1,&\text{if}\;\;\xi<-1
\end{cases}
\end{equation}
can be written as $\varrho=\prox_\phi$, where $\phi$ is the
indicator function of $[-1,1]$.
\end{example}

\begin{example}
\label{ex:1}
The rectified linear unit (ReLU) activation function \cite{Nair10} 
\begin{equation}
\varrho\colon\RR\to\RR\colon\xi\mapsto
\begin{cases}
\xi,&\text{if}\;\;\xi>0;\\
0,&\text{if}\;\;\xi\leq 0
\end{cases}
\end{equation}
can be written as $\varrho=\prox_{\phi}$, where $\phi$ is the 
indicator function of $\RP$.
\end{example}

\begin{example}
\label{ex:2}
Let $\alpha\in\rzeroun$. The parametric rectified linear unit 
activation function \cite{Heka12} is
\begin{equation}
\label{e:fence1}
\varrho\colon\RR\to\RR\colon\xi\mapsto
\begin{cases}
\xi,&\text{if}\;\;\xi> 0;\\
\alpha\xi,&\text{if}\;\;\xi\leq 0.
\end{cases}
\end{equation}
We have $\varrho=\prox_{\phi}$, where 
\begin{equation}
\phi\colon\RR\to\RX\colon\xi\mapsto
\begin{cases}
0,&\text{if}\;\;\xi>0;\\
(1/\alpha-1)\xi^2/2,
&\text{if}\;\;\xi\leq 0.
\end{cases}
\end{equation}
\end{example}
\begin{proof}
Let $\xi\in\RR$. Then $\phi'(\xi)=0$ if $\xi>0$, and 
$\phi'(\xi)=(1/\alpha-1)\xi$ if $\xi\leq0$. In turn
$(\Id+\phi')\xi=\xi$ if $\xi>0$, and 
$(\Id+\phi')(\xi)=\xi/\alpha$ if $\xi\leq0$. Hence,
$\varrho=(\Id+\phi')^{-1}$ is given by \eqref{e:fence1}.
\end{proof}

\begin{example}
\label{ex:3}
The bent identity activation function 
$\varrho\colon\RR\to\RR\colon\xi\mapsto(\xi+\sqrt{\xi^2+1}-1)/2$
satisfies $\varrho=\prox_{\phi}$, where 
\begin{equation}
\phi\colon\RR\to\RX\colon\xi\mapsto
\begin{cases}
\xi/2-\big(\ln(\xi+1/2)\big)/4,&\text{if}\;\;\xi>-1/2;\\
\pinf,&\text{if}\;\;\xi\leq-1/2.
\end{cases}
\end{equation}
\end{example}
\begin{proof}
This follows from 
\cite[Lemma~2.6 and Example~2.18]{Smms05}.
\end{proof}

\begin{example}
\label{ex:4}
The inverse square root unit activation function \cite{Carl17} is 
$\varrho\colon\RR\to\RR\colon\xi\mapsto{\xi}/{\sqrt{1+\xi^2}}$.
We have $\varrho=\prox_{\phi}$, where 
\begin{equation}
\phi\colon\RR\to\RX\colon\xi\mapsto
\begin{cases}
-\xi^2/2-\sqrt{1-\xi^2},&\text{if}\;\;|\xi|\leq 1;\\
\pinf,&\text{if}\;\;|\xi|>1.
\end{cases}
\end{equation}
\end{example}
\begin{proof}
Let $\xi\in\left]-1,1\right[=\dom\nabla\phi=\dom\partial\phi
=\ran\prox_\phi$. 
Then $\xi+\phi'(\xi)=\xi/\sqrt{1-\xi^2}$ and therefore 
$\prox_\phi=(\Id+\phi')^{-1}\colon\mu\mapsto\mu/\sqrt{1+\mu^2}$. 
\end{proof}

\begin{example}
\label{ex:14}
The inverse square root linear unit activation function 
\cite{Carl17} 
\begin{equation}
\label{e:fence2}
\varrho\colon\RR\to\RR\colon\xi\mapsto
\begin{cases}
\xi,&\text{if}\;\;\xi\geq 0;\\
\dfrac{\xi}{\sqrt{1+\xi^2}},&\text{if}\;\;\xi<0
\end{cases}
\end{equation}
can be written as $\varrho=\prox_{\phi}$, where 
\begin{equation}
\phi\colon\RR\to\RX\colon\xi\mapsto
\begin{cases}
0,&\text{if}\;\;\xi\geq 0;\\
1-\xi^2/2-\sqrt{1-\xi^2},&\text{if}\;\;-1\leq\xi<0;\\
\pinf,&\text{if}\;\;\xi<-1.
\end{cases}
\end{equation}
\end{example}
\begin{proof}
Let $\xi\in\left]-1,\pinf\right[=\dom\nabla\phi=\ran\prox_\phi$.
Then $\xi+\phi'(\xi)=\xi$ if $\xi\geq 0$, and 
$\xi+\phi'(\xi)=\xi/\sqrt{1-\xi^2}$ if $\xi<0$. Hence,
$\varrho=(\Id+\phi')^{-1}$ is given by \eqref{e:fence2}.
\end{proof}

\begin{example}
\label{ex:5}
The arctangent activation function $(2/\pi)\text{arctan}$ is the
proximity operator of 
\begin{align}
\label{e:Louie2012c}
\phi\colon\RR&\to\RX\colon\xi\mapsto
\begin{cases}
-\dfrac{2}{\pi}\ln\Big(\cos\Big(\Frac{\pi\xi}{2}\Big)\Big)-
\dfrac{1}{2}\xi^2,&\text{if}\;\;|\xi|<1;\\
\pinf,&\text{if}\;\;|\xi|\geq 1.
\end{cases}
\end{align}
\end{example}
\begin{proof}
Let $\xi\in\left]-1,1\right[=\dom\nabla\phi=
\ran\prox_\phi$.
Then $\xi+\phi'(\xi)=\text{tan}(\pi\xi/2)$ and therefore 
$\varrho=(\Id+\phi')^{-1}=(2/\pi)\text{arctan}$.
\end{proof}

\begin{example}
\label{ex:6}
The hyperbolic tangent activation function $\text{tanh}$
\cite{Lecu98} is the proximity operator of 
\begin{align}
\phi\colon\RR\to\RX\colon
\xi\mapsto
\begin{cases}
\dfrac{(1+\xi)\ln(1+\xi)+(1-\xi)\ln(1-\xi)-\xi^2}{2}
&\text{if}\;\;|\xi|<1;\\
\ln(2)-1/2&\text{if}\;\;|\xi|=1;\\
\pinf,&\text{if}\;\;|\xi|>1.
\end{cases}
\end{align}
\end{example}
\begin{proof}
Let $\xi\in\left]-1,1\right[=\dom\nabla\phi=\ran\prox_\phi$.
Then $\xi+\phi'(\xi)=\text{arctanh}(\xi)$ and therefore 
$\varrho=(\Id+\phi')^{-1}=\text{tanh}$.
\end{proof}

\begin{example}
\label{ex:7}
The unimodal sigmoid activation function \cite{Glor11}
\begin{equation}
\label{e:ex7}
\varrho\colon\RR\to\RR\colon\xi\mapsto
\frac{1}{1+e^{-\xi}}-\frac{1}{2}
\end{equation}
is the proximity operator of 
\begin{align}
\phi\colon\RR&\to\RX\nonumber\\
\xi&\mapsto
\begin{cases}
(\xi+1/2)\ln(\xi+1/2)+(1/2-\xi)\ln(1/2-\xi)-\dfrac{1}{2}
(\xi^2+1/4)
&\text{if}\;\;|\xi|<1/2;\\
-1/4,&\text{if}\;\;|\xi|=1/2;\\
\pinf,&\text{if}\;\;|\xi|>1/2.
\end{cases}
\end{align}
\end{example}
\begin{proof}
Let $\xi\in\left]-1/2,1/2\right[=\dom\nabla\phi=\ran\prox_\phi$.
Then $\xi+\phi'(\xi)=\ln((1+2\xi)/(1-2\xi))$ and therefore 
$\prox_\phi=(\Id+\phi')^{-1}\colon\mu\mapsto 
(1/2)(e^\mu-1)/(e^\mu+1)=1/(1+e^{-\mu})-1/2$.
\end{proof}

\begin{remark}
Examples~\ref{ex:6} and \ref{ex:7} are closely related in the sense
that the function of \eqref{e:ex7} can be written as
$\varrho=(1/2)\text{tanh}(\cdot/2)$.
\end{remark}

\begin{example}
\label{ex:9}
The Elliot activation function is \cite{Elli93} 
$\varrho\colon\RR\to\RR\colon\xi\mapsto{\xi}/{(1+|\xi|)}$
can be written as $\varrho=\prox_\phi$, where 
\begin{align}
\phi\colon\RR&\to\RX\nonumber\\
\xi&\mapsto
\begin{cases}
-|\xi|-\ln(1-|\xi|)-\Frac{\xi^2}{2},&\text{if}\;\;|\xi|<1;\\
\pinf,&\text{if}\;\;|\xi|\geq 1.
\end{cases}
\end{align}
\end{example}
\begin{proof}
Let $\xi\in\left]-1,1\right[=\dom\nabla\phi=\ran\prox_\phi$. 
Then $\xi+\phi'(\xi)=\xi/(1-|\xi|)$ and therefore 
$\prox_\phi=(\Id+\phi')^{-1}\colon\mu\mapsto\mu/(1+|\mu|)$. 
\end{proof}

\begin{example}
\label{ex:2013-01-30}
The inverse hyperbolic sine activation function 
$\text{arcsinh}$ is the proximity operator of 
$\phi=\text{cosh}-|\cdot|^2/2$.
\end{example}
\begin{proof}
Let $\xi\in\RR$. Then $\xi+\phi'(\xi)=\text{sinh}\,\xi$ and 
therefore $\prox_\phi=(\Id+\phi')^{-1}=\text{arcsinh}$.
\end{proof}

\begin{example}
\label{ex:24}
The logarithmic activation function \cite{Bils00}
\begin{equation}
\varrho\colon\RR\to\RR\colon\xi\mapsto\sign(\xi)\ln\big(1+|\xi|\big)
\end{equation}
is the proximity operator of 
\begin{equation}
\phi\colon\RR\to\RX
\colon\xi\mapsto e^{|\xi|}-|\xi|-1-\frac{\xi^2}{2}.
\end{equation}
\end{example}
\begin{proof}
We have $\phi'\colon\xi\mapsto\sign(\xi)(e^{|\xi|}-1)-\xi$. Hence
$(\Id+\phi')\colon\xi\mapsto\sign(\xi)(e^{|\xi|}-1)$ and, in turn,
$\prox_\phi=(\Id+\phi')^{-1}\colon\xi\mapsto\sign(\xi)\ln(1+|\xi|)$.
\end{proof}

The class of activation functions $\AL(\RR)$ has interesting 
stability properties. 

\begin{proposition}
\label{p:2}
The following hold:
\begin{enumerate}
\itemsep0mm 
\item
\label{p:2i}
Let $\alpha\in\RPP$ and $\beta\in\RPP$ be such that 
$\alpha\beta\leq 1$, and let $\varrho\in\AL(\RR)$. Then
$\alpha \varrho(\beta\cdot)\in\AL(\RR)$.
\item
\label{p:2ii}
Let $(\varrho_i)_{i\in I}$ be a finite family in $\AL(\RR)$ and let
$(\omega_i)_{i\in I}$ be real numbers in $\rzeroun$ such that
$\sum_{i\in I}\omega_i=1$. Then 
$\sum_{i\in I}\omega_i\varrho_i\in\AL(\RR)$.
\item
\label{p:2iii}
Let $\varrho_1\in\AL(\RR)$ and $\varrho_2\in\AL(\RR)$.
Then $\varrho_1\circ\varrho_2\in\AL(\RR)$.
\item
\label{p:2iv}
Let $\varrho\in\AL(\RR)$. Then $\Id-\varrho\in\AL(\RR)$.
\item
\label{p:2v}
Let $\varrho_1\in\AL(\RR)$ and $\varrho_2\in\AL(\RR)$.
Then $(\varrho_1-\varrho_2+\Id)/2\in\AL(\RR)$.
\item
\label{p:2vi}
Let $\varrho_1\in\AL(\RR)$ and $\varrho_2\in\AL(\RR)$.
Then $\varrho_1\circ(2\varrho_2-\Id)+\Id-\varrho_2\in\AL(\RR)$.
\end{enumerate}
\end{proposition}
\begin{proof}
\ref{p:2i}--\ref{p:2iii}: This follows at once from
Definition~\ref{d:1}.

\ref{p:2iv}--\ref{p:2v}:
The fact that the resulting operators are proximity operators is
established in \cite[Section~3.3]{Comb18}. The fact that they are
proximity operators of a function $\phi\in\Gamma_0(\HH)$ that is 
minimal at $0$ is equivalent to the fact that $\prox_\phi 0=0$
Lemma~\ref{l:1}\ref{l:1i}. This identity is easily 
seen to hold in each instance. 

\ref{p:2vi}: Set
$\varrho=\varrho_1\circ(2\varrho_2-\Id)+\Id-\varrho_2$. Then
$\varrho$ is firmly
nonexpansive \cite[Proposition~4.31(ii)]{Livre1}. It is therefore
increasing and nonexpansive. Finally, $\varrho(0)=0$.
\end{proof}

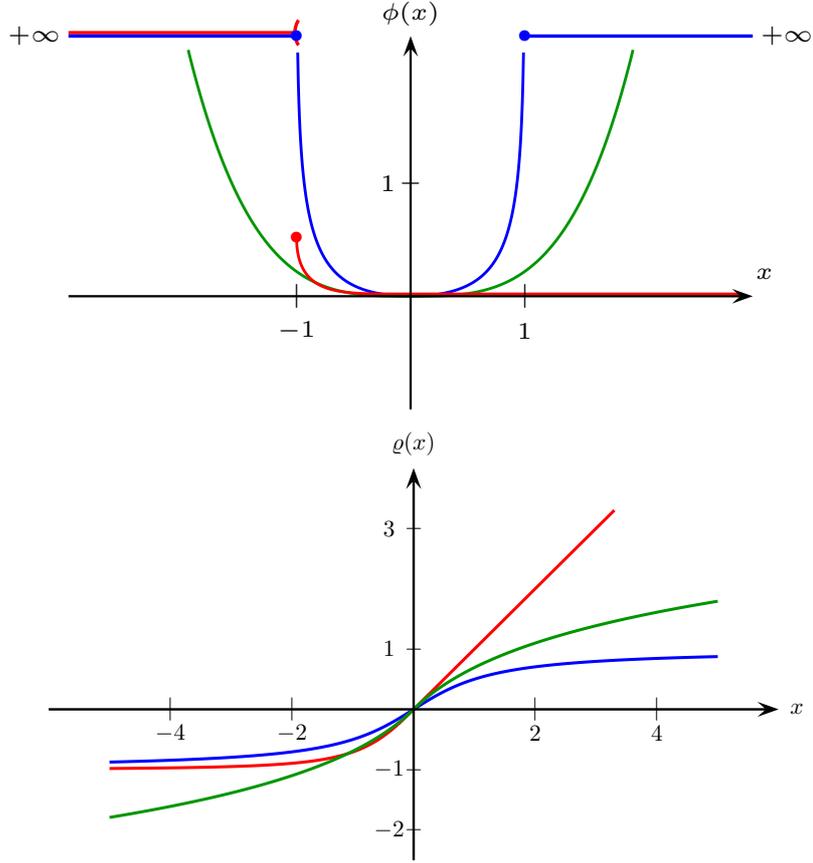
\begin{figure}[!th]
\begin{center}
\vskip -22mm
\scalebox{1.5} 
{
\begin{pspicture}(-3.5,-2.2)(3.5,3.0) 
\def\muu{3.14159}
\def\eee{2.7183}
\psplot[plotpoints=400,linewidth=0.025cm,linestyle=solid,%
algebraic,linecolor=blue]{-0.99}{0.99}%
{-(2/\muu)*ln(cos(\muu*x/2))-x^2/2}
\psplot[plotpoints=400,linewidth=0.025cm,linestyle=solid,%
algebraic,linecolor=dgreen]{-1.95}{1.95}{\eee^(abs(x))-abs(x)-1-x^2/2}
\psplot[plotpoints=400,linewidth=0.025cm,linestyle=solid,%
algebraic,linecolor=red]{-1.0}{0.0}{0.02+1-x^2/2-sqrt(1-x^2)}
\psplot[plotpoints=400,linewidth=0.025cm,linestyle=solid,%
algebraic,linecolor=red]{0.0}{2.9}{0.02}
\psline[linewidth=0.03cm,linecolor=blue,linestyle=solid]%
(1.0,2.3)(3.0,2.3)
\psline[linewidth=0.03cm,linecolor=red,linestyle=solid]{-(}%
(-3.0,2.33)(-1.0,2.33)
\psline[linewidth=0.03cm,linecolor=blue,linestyle=solid]%
(-3.0,2.3)(-1.0,2.3)
\psline[linewidth=0.02cm,arrowsize=0.05cm 4.0,%
arrowlength=1.4,arrowinset=0.4]{->}(-3,0)(3.0,0)
\psline[linewidth=0.02cm,arrowsize=0.05cm 4.0,%
arrowlength=1.4,arrowinset=0.4]{->}(0,-1)(0,2.3)
\rput(1.0,2.3){\blue\tiny$\bullet$}
\rput(-1.0,2.3){\blue\tiny$\bullet$}
\rput(-1.0,0.52){\red\tiny$\bullet$}
\rput(-1.0,-0.0){\tiny$|$}
\rput(-1.0,-0.3){\tiny$-1$}
\rput(1.0,-0.0){\tiny$|$}
\rput(1.0,-0.3){\tiny$1$}
\rput(-0.2,1.00){\tiny$1$}
\rput(-0.0,1.00){\tiny$-$}
\rput(0.0,2.5){\tiny$\phi(x)$}
\rput(3.1,0.2){\tiny$x$}
\rput(3.3,2.3){\tiny$\pinf$}
\rput(-3.3,2.3){\tiny$\pinf$}
\end{pspicture} 
}
\vskip -17mm
\scalebox{0.8} 
{
\begin{pspicture}(-6.7,-2.5)(6.6,4.8) 
\def\muu{3.141592}
\psplot[plotpoints=400,linewidth=0.05cm,linestyle=solid,%
algebraic,linecolor=blue]{-5.0}{5.0}{2*ATAN(x)/\muu}
\psplot[plotpoints=400,linewidth=0.05cm,linestyle=solid,%
algebraic,linecolor=red]{0.0}{3.3}{x}
\psplot[plotpoints=400,linewidth=0.05cm,linestyle=solid,%
algebraic,linecolor=red]{-5.0}{0.0}{x/sqrt(1+x^2)}
\psplot[plotpoints=400,linewidth=0.05cm,linestyle=solid,%
algebraic,linecolor=dgreen]{-5.0}{0.0}{-ln(1-x)}
\psplot[plotpoints=400,linewidth=0.05cm,linestyle=solid,%
algebraic,linecolor=dgreen]{0.0}{5.0}{ln(1+x)}
\psline[linewidth=0.04cm,arrowsize=0.09cm 4.0,%
arrowlength=1.4,arrowinset=0.4]{->}(-6,0)(6.0,0)
\psline[linewidth=0.04cm,arrowsize=0.09cm 4.0,%
arrowlength=1.4,arrowinset=0.4]{->}(0,-2.5)(0,4.0)
\rput(-2.00,-0.0){$|$}
\rput(-4.00,-0.0){$|$}
\rput(-2.00,-0.4){$-2$}
\rput(-4.00,-0.4){$-4$}
\rput(2.00,-0.0){$|$}
\rput(4.00,-0.0){$|$}
\rput(2.00,-0.4){$2$}
\rput(-0.4,3.00){$3$}
\rput(-0.0,3.00){$-$}
\rput(-0.4,-2.00){$-2$}
\rput(-0.0,-2.00){$-$}
\rput(-0.4,-1.00){$-1$}
\rput(-0.0,-1.00){$-$}
\rput(-0.4,1.00){$1$}
\rput(-0.0,1.00){$-$}
\rput(4.00,-0.4){$4$}
\rput(0.0,4.4){$\varrho(x)$}
\rput(6.3,0){$x$}
\end{pspicture} 
}
\end{center}
\vskip 0mm
\caption{The function $\phi$ (top) and the corrresponding proximal
activation function (bottom) $\varrho$ in Proposition~\ref{p:1}.
Example~\ref{ex:14} is in red, Example~\ref{ex:5} is in blue, 
Example~\ref{ex:24} is in green.}
\label{fig:2}
\end{figure}

\begin{remark}
Using Proposition~\ref{p:2}, the above examples can be
combined to obtain additional activation functions. For instance,
it follows from Example~\ref{ex:0} and 
Proposition~\ref{p:2}\ref{p:2iv} that the soft thresholder
\begin{equation}
\varrho\colon\RR\to\RR\colon\xi\mapsto
\begin{cases}
\xi-1,&\text{if}\;\;\xi>1;\\
0,&\text{if}\;\;-1\leq\xi\leq 1;\\
\xi+1,&\text{if}\;\;\xi<-1
\end{cases}
\end{equation}
belongs to $\AL(\RR)$. It was proposed as an activation function 
in \cite{Zhan01}.
\end{remark}

\subsection{Activation operators}

In Section~\ref{sec:31}, we have described activation functions
which model neuronal activity in terms of a scalar function. In
this section, we extend this notion to more general 
activation operators. 

\begin{definition}
\label{d:2}
Let $\HH$ be a real Hilbert space and let $R\colon\HH\to\HH$.
Then $R$ belongs to the class $\AL(\HH)$ if there exists a 
function $\varphi\in\Gamma_0(\HH)$ which is
minimal at the zero vector and such that $R=\prox_\varphi$.
\end{definition}

Property~\ref{p:5i} below shows that activation operators in 
$\AL(\HH)$ have strong stability properties. On the other hand, the
boundedness property~\ref{p:5iii} is important in neural
network-based functional approximation \cite{Cybe89,Funa89}.

\begin{proposition}
\label{p:5}
Let $\HH$ be a real Hilbert space
and let $R\in\AL(\HH)$. Then the following hold:
\begin{enumerate}
\itemsep0mm 
\item
\label{p:5i-}
$R0=0$.
\item
\label{p:5i}
Let $x$ and $y$ be in $\HH$. Then
$\|Rx-Ry\|^2\leq\|x-y\|^2-\|x-y-Rx+Ry\|^2$.
\item
\label{p:5ii}
Let $x\in\HH$. Then $\|Rx\|\leq\|x\|$.
\item
\label{p:5iii}
Let $\varphi\in\Gamma_0(\HH)$ be such that $R=\prox_\varphi$. 
Then $\ran R$ is bounded if and only if $\dom\varphi$ is bounded. 
\end{enumerate}
\end{proposition}
\begin{proof}
\ref{p:5i-}: This follows from Lemma~\ref{l:1}\ref{l:1i}.

\ref{p:5i}: This follows from the firm nonexpansiveness 
of proximity operators \cite[Proposition~12.28]{Livre1}.

\ref{p:5ii}: Set $y=0$ in \ref{p:5i} and use \ref{p:5i-}.

\ref{p:5iii}: We have $\ran R=\ran(\Id+\partial\varphi)^{-1}
=\dom(\Id+\partial\varphi)=\dom\partial\varphi$. On the other
hand, $\dom\partial\varphi$ is a dense subset of $\dom\varphi$
\cite[Corollary~16.39]{Livre1}.
\end{proof}

\begin{proposition}
\label{p:3}
Let $\HH$ and $\GG$ be real Hilbert spaces. 
Then the following hold:
\begin{enumerate}
\itemsep0mm 
\item
\label{p:3i}
Let $L\in\BL(\HH,\GG)$ be such 
that $\|L\|\leq 1$ and let $R\in\AL(\HH)$. Then 
$L^*\circ R\circ L\in\AL(\HH)$.
\item
\label{p:3ii}
Let $(R_i)_{i\in I}$ be a finite family in $\AL(\HH)$ and let
$(\omega_i)_{i\in I}$ be real numbers in $\rzeroun$ such that
$\sum_{i\in I}\omega_i=1$. Then 
$\sum_{i\in I}\omega_iR_i\in\AL(\HH)$.
\item
\label{p:3iv}
Let $R\in\AL(\HH)$. Then $\Id-R\in\AL(\HH)$.
\item
\label{p:3v}
Let $R_1\in\AL(\HH)$ and $R_2\in\AL(\HH)$.
Then $(R_1-R_2+\Id)/2\in\AL(\HH)$.
\end{enumerate}
\end{proposition}
\begin{proof}
The fact that the resulting operators are proximity operators is
established in \cite[Section~3.3]{Comb18}. In addition, $0$ is
clearly a fixed point of the resulting operators. In view of
Lemma~\ref{l:1}\ref{l:1i}, the proof is complete.  
\end{proof}

\begin{example}
The softmax activation operator \cite{Brid90} is 
\begin{equation}
R\colon\RR^N\to\RR^N\colon (\xi_k)_{1\leq k\leq N}\mapsto
\left({\exp(\xi_k)}\left/{\displaystyle\sum_{j=1}^N 
\exp(\xi_j)}\right.\right)_{1\leq k\leq N}-u,
\end{equation}
where $u=(1,\ldots,1)/N\in \RR^N$.
We have $R=\prox_\varphi$, where
$\varphi=\psi(\cdot+u)+\scal{\cdot}{u}$ and
\begin{align}
\psi\colon\RR^N&\to\RX\nonumber\\
(\xi_k)_{1\leq k\leq N}&\mapsto
\begin{cases}
\displaystyle\sum_{k=1}^N\bigg(\xi_k\ln\xi_k-
\frac{\xi_k^2}{2}\bigg),
&\displaystyle\text{if}\;\;(\xi_k)_{1\leq i\leq N}\in [0,1]^N\;
\;\text{and}\;\;\sum_{k=1}^N\xi_k=1;\\
\pinf,&\text{otherwise},
\end{cases}
\end{align}
with the convention $0\ln 0=0$.
\end{example}
\begin{proof}
Set
\begin{align}
g\colon\RR^N&\to\RX\nonumber\\
(\xi_k)_{1\leq k\leq N}&\mapsto
\begin{cases}
\Sum_{k=1}^N\xi_k\ln\xi_k,
&\text{if}\;\;(\xi_k)_{1\leq k\leq N}\in[0,1]^N\;
\;\text{and}\;\;\Sum_{k=1}^N\xi_k=1;\\
\pinf,&\text{otherwise}.
\end{cases}
\end{align}
Then $\psi=g-\|\cdot\|^2/2$ and \cite[Section~16]{Rock70} asserts
that
\begin{equation}
g^*\colon\RR^N\to\RR\colon (\xi_k)_{1\leq k\leq N}\mapsto
\ln\left(\Sum_{k=1}^N\exp(\xi_k)\right).
\end{equation}
Since $\nabla g^*=R+u$, according to
Lemma~\ref{l:1}\ref{l:1ii}, $R=\prox_{\psi}-u$. We 
complete the proof by invoking the shift properties of 
proximity operators \cite[Proposition~24.8(iii)]{Livre1}.
\end{proof}

Separable activation operators supply another important instance 
of activation operators.

\begin{proposition}
\label{p:4}
Let $\HH$ be a separable real Hilbert space, let
$(e_k)_{k\in\bK\subset\NN}$ be an orthonormal basis of 
$\HH$, and let $(\phi_k)_{k\in\bK}$ be a family of functions in
$\Gamma_0(\RR)$ such that
$(\forall k\in\bK)$ $\phi_k\geq\phi_k(0)=0$. Define
\begin{equation}
R\colon\HH\to\HH\colon
x\mapsto\sum_{k\in\bK}\big(\prox_{\phi_k}\scal{x}{e_k}\big)e_k.
\end{equation}
Then $R\in\AL(\HH)$.
\end{proposition}
\begin{proof}
The fact that $R$ is the proximity operator of the $\Gamma_0(\HH)$
function $\varphi\colon
x\mapsto\sum_{k\in\bK}\phi_k(\scal{x}{e_k})$ is
established in \cite[Example~2.19]{Smms05}.
In addition, it is clear that $\varphi$ is minimal at $0$.
\end{proof}

\section{Compositions of firmly nonexpansive and affine operators}
\label{sec:2}

Our analysis will revolve around the following property
for a family of linear operators $(W_i)_{1\leq i\leq m+1}$.

\begin{condition}
\label{o:1}
Let $m\geq 0$ be an integer, let 
$(\HH_i)_{0\leq i\leq m}$ be real Hilbert spaces, set
$\HH_{m+1}=\HH_0$, and let $\alpha\in [1/2,1]$. For every
$i\in\{1,\ldots,m+1\}$, let $W_i\in\BL(\HH_{i-1},\HH_i)$ and set
\begin{equation}
\label{e:defLi}
L_i\colon\HH_0\times\cdots\times\HH_{i-1}\to\HH_i\colon
(x_k)_{0\leq k\leq i-1}\mapsto
\sum_{k=0}^{i-1}\big(W_i\circ\cdots\circ W_{k+1}\big) x_k.
\end{equation}
It is required that, 
for every $\boldsymbol{x}=(x_i)_{0\leq i\leq m} 
\in\HH_0\times\cdots\times\HH_m$ such that
\begin{equation}
\label{e:alphastabcondh} 
(\forall i\in\{0,\ldots,m\})\quad\|x_i\|\leq
\begin{cases}
1,&\text{if}\;\;i=0;\\
\|L_i(x_0,\ldots,x_{i-1})\|,&\text{if}\;\;i\geq 1,
\end{cases}
\end{equation}
there holds
\begin{equation}
\label{e:alphastabcondn} 
\|L_{m+1}\boldsymbol{x}-2^{m+1}(1-\alpha)x_0\|+
\|L_{m+1}\boldsymbol{x}\|\leq2^{m+1}\alpha \|x_0\|.
\end{equation}
\end{condition}

\begin{remark}
\label{re:alphastabn}
In Condition~\ref{o:1}, we take $\alpha\geq 1/2$ because, if 
$\boldsymbol{x}=(x_i)_{0\leq i\leq m}\in(\HH_0\smallsetminus\{0\})
\times\HH_1\times\cdots\times\HH_m$ satisfies 
\eqref{e:alphastabcondn}, then
$2^{m+1}(1-\alpha)\|x_0\|\leq
\|L_{m+1}\boldsymbol{x}-2^{m+1}(1-\alpha)x_0\|+
\|L_{m+1}\boldsymbol{x}\|\leq 2^{m+1}\alpha\|x_0\|$.
\end{remark}

We establish some preliminary results before providing properties
that imply Condition~\ref{o:1}.

\begin{lemma}
\label{l:11}
Let $m\geq 1$ be an integer, let $(\HH_i)_{0\leq i\leq m}$ be real
Hilbert spaces, and set $\theta_0=1$. For every
$i\in\{1,\ldots,m\}$, let $W_i\in\BL(\HH_{i-1},\HH_i)$ and set
\begin{multline}
\label{e:defthetaell}
\theta_i=\|W_i\circ\cdots\circ W_1\|\\
+\sum_{k=1}^{i-1}\sum_{1\leq j_1<\ldots<j_k\leq i-1}
\|W_i\circ\cdots\circ W_{j_k+1}\|\,
\|W_{j_k}\circ\cdots\circ W_{j_{k-1}+1}\|\cdots 
\|W_{j_1}\circ\cdots\circ W_1\|.
\end{multline}
Let $(x_i)_{0\leq i\leq m} \in\HH_0\times\cdots\times\HH_m$ be
such that \eqref{e:alphastabcondh} is satisfied. Then the 
following hold:
\begin{enumerate}
\itemsep0mm 
\item
\label{l:11i}
$(\forall i\in\{1,\ldots,m\})$ $\theta_i=\sum_{k=0}^{i-1}
\theta_k\|W_i\circ\cdots\circ W_{k+1}\|$.
\item
\label{l:11ii}
$(\forall i\in\{1,\ldots,m\})$ $\|x_i\|\leq\theta_i\|x_0\|$.
\end{enumerate}
\end{lemma}
\begin{proof}
\ref{l:11i}: This follows recursively from \eqref{e:defthetaell}.

\ref{l:11ii}:
For every $i\in\{1,\ldots,m\}$, let $L_i$ be as in \eqref{e:defLi}. 
We proceed by induction on $m$. We first observe that the
inequality is satisfied if $m=1$ since
$\|x_1\|\leq\|L_1x_0\|=\|W_1x_0\|\leq\|W_1\|\,\|x_0\|
=\theta_1\|x_0\|$. Now assume that $m\geq 2$ and that the
inequalities hold for $(x_1,\ldots,x_{m-1})$. Then, since
\ref{l:11i} yields
\begin{equation}
\theta_m=\|W_m\circ\cdots\circ W_1\|+\sum_{k=1}^{m-1}
\theta_k\|W_m\circ\cdots\circ W_{k+1}\|,
\end{equation}
we obtain
\begin{align}
\|x_m\|
\leq\|L_m(x_0,\ldots,x_{m-1})\|&=\bigg\|\sum_{k=0}^{m-1}
(W_m\circ\cdots\circ W_{k+1})x_k\bigg\|
\nonumber\\
&\leq\sum_{k=0}^{m-1}\|W_m\circ\cdots\circ W_{k+1}\|\,
\|x_k\|\nonumber\\
&\leq\bigg(\|W_m\circ\cdots\circ W_1\|+\sum_{k=1}^{m-1}\theta_k
\|W_m\circ\cdots\circ W_{k+1}\|\bigg)\|x_0\|\nonumber\\
&=\theta_m\|x_0\|,
\end{align}
which concludes the proof.
\end{proof}

\begin{lemma}
\label{l:8}
Let $\HH$ be a real Hilbert space, and let $x$ and $y$ be in $\HH$.
Then
\begin{equation}
\|x\|\,\|y\|-\scal{x}{y}\leq(\|x\|+\|y\|-\|x+y\|)(\|x\|+\|y\|).
\end{equation}
\end{lemma}
\begin{proof}
Since $\|x+y\|^2-2\|x+y\|(\|x\|+\|y\|)+(\|x\|+\|y\|)^2\geq 0$,
we have 
\begin{align}
&\hskip -8mm\|x\|^2+\|y\|^2+\scal{x}{y}+\|x\|\,\|y\|\nonumber\\
&=\|x\|^2+\|y\|^2
+\frac{\|x+y\|^2-\|x\|^2-\|y\|^2}{2}
+\frac{(\|x\|+\|y\|)^2-\|x\|^2-\|y\|^2}{2}\nonumber\\
&=\frac{\|x+y\|^2+(\|x\|+\|y\|)^2}{2}\nonumber\\
&\geq\|x+y\|(\|x\|+\|y\|), 
\end{align}
as claimed.
\end{proof}

\begin{notation}
\label{n:1}
Let $m\geq 0$ be an integer, and let 
$(\HH_i)_{0\leq i\leq m}$ be
real Hilbert spaces. Let $\XXX$ be the standard vector space 
$\HH_0\times\cdots\times\HH_m$ equipped with the norm
$\|\cdot\|_{\XXX}\colon\boldsymbol{x}=(x_i)_{0\leq i\leq m}\mapsto
\max_{0\leq i\leq m}\|x_i\|$ and let $\YYY$ be the standard 
vector space $\HH_{0}\times\HH_{0}$ equipped with the norm 
$\|\cdot\|_{\YYY}\colon\boldsymbol{y}=(y_1,y_2)
\mapsto\|y_1\|+\|y_2\|$. Henceforth, the norm of 
$\boldsymbol{M}\in\BL(\XXX,\YYY)$ is denoted by 
$\|\boldsymbol{M}\|_{\XXX,\YYY}$.
\end{notation}

\begin{proposition}
\label{p:9}
Let $m\geq 0$ be an integer, let 
$(\HH_i)_{0\leq i\leq m}$ be nonzero real Hilbert spaces, set
$\HH_{m+1}=\HH_0$, and use Notation~\ref{n:1}. For every
$i\in\{1,\ldots,m+1\}$, let $W_i\in\BL(\HH_{i-1},\HH_i)$. Further,
let $\alpha\in[1/2,1]$, let $\theta_0=1$, let 
$(\theta_i)_{1\leq i\leq m+1}$ be as in \eqref{e:defthetaell}, 
and set
\begin{subequations}
\begin{empheq}[left={\empheqlbrace\,}]{align}
&W=W_{m+1}\circ\cdots\circ W_1
\label{e:2.21a}\\
&\displaystyle\mu=\inf_{x\in\HH_0,\,\|x\|=1}\scal{Wx}{x}
\label{e:2.21b}\\
&\displaystyle M\colon\XXX\to\HH_0\colon\boldsymbol{x}\mapsto
\sum_{i=0}^m\theta_i(W_{m+1}\circ\cdots\circ W_{i+1})x_i 
\label{e:2.21d}\\
&\boldsymbol{M}\colon\XXX\to\YYY\colon\boldsymbol{x}\mapsto
\Frac{1}{2^{m+1}\alpha}\big(M\boldsymbol{x}-
2^{m+1}(1-\alpha)x_0,M\boldsymbol{x}\big).
\label{e:2.21e}
\end{empheq}
\end{subequations}
Suppose that one of the following holds:
\begin{enumerate}
\itemsep0mm 
\item 
\label{p:9i}
There exists $i\in\{1,\ldots,m+1\}$ such that $W_i=0$.
\item 
\label{p:9ii} 
$\|\boldsymbol{M}\|_{\XXX,\YYY}\leq 1$.
\item 
\label{p:9iii}
$\|W-2^{m+1}(1-\alpha)\Id\|-\|W\|+2\theta_{m+1}\leq 2^{m+1}\alpha$.
\item 
\label{p:9iv}
$\alpha\neq 1$, for every $i\in\{1,\ldots,m+1\}$ $W_i\neq 0$, and
there exists $\eta\in [0,\alpha/((1-\alpha)\theta_{m+1})]$
such that
\begin{equation}
\begin{cases}
\theta_{m+1}\leq 2^{m+1}\alpha\\
\alpha\theta_{m+1}+(1-\alpha)(\|\Id-\eta W\|-\eta\|W\|)
(\theta_{m+1}-\|W\|)\leq 2^m(2\alpha-1)+(1-\alpha)\mu.
\end{cases}
\end{equation}
\end{enumerate}
Then $(W_i)_{1\leq i\leq m+1}$ satisfies Condition~\ref{o:1}.
\end{proposition}
\begin{proof}
We use the operators $(L_i)_{1\leq i\leq m+1}$ introduced in 
Condition~\ref{o:1}. Per Notation~\ref{n:1} and \eqref{e:2.21e},
\begin{equation}
\label{e:fz70} 
\sup_{\substack{\boldsymbol{y}\in\XXX\\
{\underset{0\leq i\leq m}{\max}\|y_i\|\leq 1}}}
\frac{\|M\boldsymbol{y}-2^{m+1}(1-\alpha)y_0\|
+\|M\boldsymbol{y}\|}{2^{m+1}\alpha}
=\sup_{\substack{\boldsymbol{y}\in\XXX\\ 
\|\boldsymbol{y}\|_{\XXX}\leq 1}}
\|\boldsymbol{M}\boldsymbol{y}\|_{\YYY}= 
\|\boldsymbol{M}\|_{\XXX,\YYY}
\end{equation}
and therefore
\begin{equation}
\label{e:fz73}
(\forall\boldsymbol{y}\in\XXX)\quad
\max_{0\leq i\leq m}\|y_i\|\leq 1\quad\Rightarrow\quad
\|M\boldsymbol{y}-2^{m+1}(1-\alpha)y_0\|+
\|M\boldsymbol{y}\|\leq 2^{m+1}\alpha\|\boldsymbol{M}\|_{\XXX,\YYY}.
\end{equation}
Now let $\boldsymbol{x}\in\XXX$ be such that
\begin{equation} 
\label{e:fz71} 
(\forall i\in\{0,\ldots,m\})\quad\|x_i\|\leq
\begin{cases}
1,&\text{if}\;\;i=0;\\
\|L_i(x_0,\ldots,x_{i-1})\|,&\text{if}\;\;i\geq 1.
\end{cases}
\end{equation}

\ref{p:9i}:
We assume that $m\geq 1$. For every $k\in\{i,\ldots,m\}$, it
follows from \eqref{e:defthetaell} that $\theta_k=0$ and in turn
from Lemma~\ref{l:11}\ref{l:11ii} and \eqref{e:fz71} that $x_k=0$.
Therefore,
\begin{equation}
L_{m+1}\boldsymbol{x}=
\sum_{k=0}^m(W_{m+1}\circ\cdots\circ W_{k+1})x_k=
\sum_{k=0}^{i-1}(W_{m+1}\circ\cdots\circ W_{k+1})x_k=0,
\end{equation}
and \eqref{e:alphastabcondn} clearly holds.

\ref{p:9ii}:
In view of \ref{p:9i}, we assume that, if $m\geq 1$, 
$(\forall i\in\{1,\ldots,m\})$ $W_i\neq 0$. We then derive from
\eqref{e:defthetaell} that $(\forall i\in\{1,\ldots,m\})$ 
$\theta_i\geq\prod_{k=1}^i\|W_k\|>0$. 
If $x_0=0$, \eqref{e:alphastabcondn} trivially 
follows from Lemma~\ref{l:11}\ref{l:11ii}, we therefore assume
otherwise. Now set
\begin{equation}
\label{e:uivithetai}
(\forall i\in\{0,\ldots,m\})\quad y_i=
\dfrac{x_i}{\theta_i\|x_0\|}.
\end{equation}
According to Lemma~\ref{l:11}\ref{l:11ii}, 
$(\forall i\in\{0,\ldots,m\})$ $\|y_i\|\leq 1$.
On the other hand, it follows from \eqref{e:2.21d},
\eqref{e:uivithetai}, and \eqref{e:defLi} that 
$M\boldsymbol{y}=L_{m+1}\boldsymbol{x}/\|x_0\|$.
Altogether, we deduce from \eqref{e:fz73} that
\eqref{e:alphastabcondn} holds. 

\ref{p:9iii}$\Rightarrow$\ref{p:9ii}:
Take $\boldsymbol{y}\in\XXX$ such that 
$\|\boldsymbol{y}\|_{\XXX}\leq 1$. Then it follows from 
\eqref{e:2.21d} and 
Lemma~\ref{l:11}\ref{l:11i} that 
\begin{align}
&\hskip -6mm\|M\boldsymbol{y}-2^{m+1}(1-\alpha)y_0\|
+\|M\boldsymbol{y}\|\nonumber\\
&\leq\|W-2^{m+1}(1-\alpha)\Id\|\,
\|y_0\|+\|W\|\,\|y_0\|+2\sum_{i=1}^m\theta_i\|W_{m+1}\circ
\cdots\circ W_{i+1}\|\,\|y_i\|\nonumber\\
&\leq\|W-2^{m+1}(1-\alpha)\Id\|-\|W\|+2\theta_{m+1}\nonumber\\
&\leq 2^{m+1}\alpha.
\end{align}
In turn, \eqref{e:fz70} yields
$\|\boldsymbol{M}\|_{\XXX,\YYY}\leq 1$.

\ref{p:9iv}$\Rightarrow$\ref{p:9ii}:
Let $\boldsymbol{y}=(y_0,\ldots,y_m)\in\XXX$ be
such that $\|y_0\|=\cdots=\|y_m\|=1$, and set 
\begin{equation}
\label{e:fz86}
u=
\begin{cases}
\displaystyle\sum_{i=1}^m\theta_i(W_{m+1}\circ\cdots
\circ W_{i+1})y_i,&\text{if}\;\;m\neq 0;\\
0, &\text{if}\;\;m=0.
\end{cases} 
\end{equation}
The assumptions and \eqref{e:2.21b} imply that
\begin{equation}
\label{e:revereng}
\begin{cases}
\eta\theta_{m+1}\leq\alpha/(1-\alpha)\\
\theta_{m+1}\leq 2^{m+1}\alpha\\
\alpha\theta_{m+1}+(1-\alpha) (\|\Id-\eta W\|-\eta \|W\|)
(\theta_{m+1}-\|W\|)\\
\hskip 73mm \leq 2^m(2\alpha-1)+(1-\alpha)\scal{Wy_0}{y_0}.
\end{cases}
\end{equation}
On the other hand, 
\begin{align}
\label{e:fz62}
&\hskip -6mm
\alpha\|Wy_0 +u\big\|-(1-\alpha) \scal{y_0}{u}\nonumber\\
&=\alpha \|Wy_0+u\big\|
-(1-\alpha)\scal{\eta Wy_0+(\Id-\eta W)y_0}{u}\nonumber\\
&\leq\alpha\|Wy_0+u\big\|-\eta(1-\alpha)\scal{Wy_0}{u}
+(1-\alpha)\|(\Id-\eta W)y_0\|\,\|u\|.
\end{align}
Since, by Lemma~\ref{l:11}\ref{l:11i} and \eqref{e:revereng},
\begin{equation}
\label{e:mago}
\eta\sum_{i=0}^m\theta_i\|W_{m+1}\circ\cdots\circ W_{i+1}\|
=\eta\theta_{m+1}\leq\frac{\alpha}{1-\alpha},
\end{equation}
we deduce from \eqref{e:fz86} that
\begin{equation}
\label{e:Luwaaegen}
\eta(1-\alpha)(\|Wy_0\|+\|u\|)\leq\alpha.
\end{equation}
However, by Lemma~\ref{l:8},
\begin{equation}
\label{e:fz67}
\|Wy_0\|\,\|u\|-\scal{Wy_0}{u}\leq
(\|Wy_0\|+\|u\|-\|Wy_0+u\|)(\|Wy_0\|+\|u\|).
\end{equation}
In view of \eqref{e:Luwaaegen}, this yields 
\begin{equation}
\label{e:fz65}
\eta(1-\alpha)\big(\|Wy_0\|\,\|u\|-\scal{Wy_0}{u}\big)\leq
\alpha(\|Wy_0\|+\|u\|-\|Wy_0+u\|),
\end{equation}
that is,
\begin{equation}
\label{e:fz66}
\alpha\|Wy_0+u\|-\eta(1-\alpha)\scal{Wy_0}{u}\leq
\alpha(\|Wy_0\|+\|u\|)-\eta(1-\alpha)\|Wy_0\|\,\|u\|.
\end{equation}
Therefore, since \eqref{e:Luwaaegen} implies that 
$\alpha-\eta(1-\alpha)\|u\|\geq 0$, it results from
\eqref{e:fz62} that
\begin{align}
\label{e:fz89}
&\hskip -6mm\alpha\|Wy_0+u\big\|-(1-\alpha)\scal{y_0}{u}\nonumber\\
&\leq\alpha(\|Wy_0\|+\|u\|)-\eta(1-\alpha)\|Wy_0\|\,\|u\|
+(1-\alpha)\|(\Id-\eta W)y_0\|\,\|u\|\nonumber\\
&=\alpha\|u\|+(\alpha-\eta(1-\alpha)\|u\|)\|Wy_0\|
+(1-\alpha)\|(\Id-\eta W)y_0\|\,\|u\|\nonumber\\
&\leq\alpha\|u\|+(\alpha-\eta(1-\alpha)\|u\|)\|W\|
+(1-\alpha)\|(\Id-\eta W)y_0\|\,\|u\|\nonumber\\
&=\alpha\|W\|+\big(\alpha-\eta(1-\alpha)\|W\|\big)\|u\|
+(1-\alpha)\|\Id-\eta W\|\,\|u\|.
\end{align}
However, since \eqref{e:mago} implies that
$\alpha-\eta(1-\alpha)\|W\|\geq 0$, while 
\eqref{e:fz86} implies that $\|u\|\leq\theta_{m+1}-\|W\|$,
we derive from \eqref{e:fz89} that
\begin{align}
\label{e:fz77}
&\hskip -6mm\alpha\|Wy_0+u\big\|-(1-\alpha)\scal{y_0}{u}\nonumber\\
&\leq\alpha\|W\|+\big(\alpha-\eta(1-\alpha)\|W\|\big)
\big(\theta_{m+1}-\|W\|\big)+
(1-\alpha)\|\Id-\eta W\|\,(\theta_{m+1}-\|W\|).
\end{align}
We also have 
\begin{equation}
\label{e:fz74}
\|Wy_0+u\|\leq\|W\|+\|u\|\leq\theta_{m+1}.
\end{equation}
Hence, using \eqref{e:fz77}, \eqref{e:fz74}, \eqref{e:2.21d}, 
\eqref{e:2.21a}, and \eqref{e:2.21e} we obtain
\begin{align}  
\label{e:fz64}
\eqref{e:revereng}\quad 
&\Rightarrow\quad 
\begin{cases}
\|Wy_0+u\|\leq 2^{m+1}\alpha\\
\alpha\|W y_0+u\big\|-(1-\alpha) \scal{y_0}{W y_0+u}
\leq 2^m(2\alpha-1)
\end{cases}\nonumber\\
&\Leftrightarrow \quad
\begin{cases}
\|M\boldsymbol{y}\|\leq2^{m+1}\alpha\\
\alpha\|M\boldsymbol{y}\big\|-(1-\alpha)\scal{y_0}{M\boldsymbol{y}}
\leq 2^m\big(\alpha^{2}-(1-\alpha)^{2}\big)
\end{cases}\nonumber\\
&\Leftrightarrow \quad 
\begin{cases}
\|M\boldsymbol{y}\|\leq2^{m+1}\alpha\\
\big\|M\boldsymbol{y}-2^{m+1}(1-\alpha)y_0\|^2\leq
\big(2^{m+1}\alpha-\|M\boldsymbol{y}\big\|\big)^{2}
\end{cases}\nonumber\\
&\Leftrightarrow\quad\|M\boldsymbol{y}-2^{m+1}(1-\alpha)y_0\|
+\|M\boldsymbol{y}\|\leq 2^{m+1}\alpha\nonumber\\
&\Leftrightarrow\quad\|\boldsymbol{M}\boldsymbol{y}\|_{\YYY}\leq 1.
\end{align}
Now set $\boldsymbol{C}=\menge{\boldsymbol{y}\in\XXX}
{\|y_0\|=\cdots=\|y_m\|=1}$.
Then, in view of \eqref{e:fz70}, \eqref{e:fz64}, and
\cite[Proposition~11.1(ii)]{Livre1}, we conclude that 
$\|\boldsymbol{M}\|_{\XXX,\YYY}
=\sup_{\boldsymbol{y}\in\conv\boldsymbol{C}}
\|\boldsymbol{M}\boldsymbol{y}\|_{\YYY}
=\sup_{\boldsymbol{y}\in\boldsymbol{C}}
\|\boldsymbol{M}\boldsymbol{y}\|_{\YYY}
\leq 1$.
\end{proof}

The next result establishes a link between deep neural network 
structures and the operators introduced in \eqref{e:defLi}.

\begin{lemma}
\label{l:peniblegenb}
Let $m\geq 1$ be an integer and let 
$(\HH_i)_{0\leq i\leq m+1}$ be nonzero real Hilbert spaces. For
every $i\in\{1,\ldots,m+1\}$, let $W_i\in\BL(\HH_{i-1},\HH_i)$ and
let $L_i$ be as in \eqref{e:defLi}. Further, for every 
$i\in\{1,\ldots,m\}$, let $P_i\colon\HH_i\to\HH_i$ be firmly 
nonexpansive. Set 
\begin{equation}
\label{e:fz81}
T_m=W_{m+1}\circ P_m\circ W_m\circ\cdots\circ P_1\circ W_1, 
\end{equation}
let $x$ and $y$ be distinct points in $\HH_0$, and set
$v_0=(x-y)/\|x-y\|$. Then there exists 
$(v_1,\ldots,v_m)\in\HH_1\times\cdots\times\HH_m$
such that
\begin{equation}
\label{e:fz76}
\begin{cases}
(\forall i\in\{1,\ldots,m\})\quad\|v_i\|
\leq\|L_i(v_0,\ldots,v_{i-1})\|\\[2mm]
\dfrac{2^m (T_mx-T_my)}{\|x-y\|}=L_{m+1}(v_0,\ldots,v_m).
\end{cases}
\end{equation}
\end{lemma}
\begin{proof}
For every $i\in\{1,\ldots,m\}$, since $P_i$ is firmly nonexpansive, 
there exists a nonexpansive operator $Q_i\colon\HH_i\to\HH_i$ such 
that
\begin{equation}
\label{e:fz79}
P_i=\frac{\Id+Q_i}{2}.
\end{equation}
We proceed by induction on $m$. Suppose that $m=1$ and set
\begin{equation}
v_1=\frac{Q_1(W_1x)-Q_1(W_1y)}{\|x-y\|},
\end{equation}
which implies that
$\|v_1\|\leq{\|W_1(x-y)\|}/{\|x-y\|}=\|L_1 v_0\|$. Then
\begin{align}
2(T_1x-T_1y)
&=(W_2\circ W_1)(x-y)+(W_2\circ Q_1\circ W_1)x-
(W_2\circ Q_1\circ W_1)y\nonumber\\
&=\|x-y\|\big((W_2\circ W_1)v_0+W_2v_1)\big).
\end{align}
Thus, \eqref{e:fz76} holds for $m=1$. Next, we assume that $m>1$
and that there exists 
$(v_1,\ldots,v_{m-1})\in\HH_1\times\cdots\times\HH_{m-1}$ such that
\begin{equation}
\label{e:fz75}
\begin{cases}
(\forall i\in\{1,\ldots,m-1\})\quad\|v_i\|
\leq\|L_i(v_0,\ldots,v_{i-1})\|\\[2mm]
\dfrac{2^{m-1}
\big(T_{m-1}x-T_{m-1}y\big)}{\|x-y\|}=L_m(v_0,\ldots,v_{m-1}),
\end{cases}
\end{equation}
and we set
\begin{equation} 
\label{e:fz80}
v_m=\frac{2^{m-1}\big((Q_m\circ T_{m-1})x-
(Q_m\circ T_{m-1})y\big)}{\|x-y\|}.
\end{equation}
Then \eqref{e:fz81}, \eqref{e:fz79}, and \eqref{e:fz75} yield
\begin{align}
\label{e:formmagicdebgen}
&\hskip -7mm T_mx-T_my\nonumber\\
&=\frac{(W_{m+1}\circ T_{m-1})x-(W_{m+1}\circ T_{m-1})y}{2}
+\frac{(W_{m+1}\circ Q_m\circ T_{m-1})x-
(W_{m+1}\circ Q_m\circ T_{m-1})y}{2}\nonumber\\
&=\frac{\|x-y\|}{2^m}\big((W_{m+1}\circ L_m)(v_0,\ldots,v_{m-1})
+W_{m+1}v_m\big)\nonumber\\
&=\frac{\|x-y\|}{2^m}L_{m+1}(v_0,\ldots,v_m).
\end{align}
In addition, it follows from \eqref{e:fz75} and \eqref{e:fz80} that
\begin{equation}
\label{e:thetamineq}
\|v_m\| \leq\frac{2^{m-1} \|T_{m-1}x-
T_{m-1}y\|}{\|x-y\|}=\|L_m(v_0,\ldots,v_{m-1})\|,
\end{equation}
which completes the proof.
\end{proof}

We now establish connections between Condition~\ref{o:1} for
linear operators and the concept of averagedness for composite
nonlinear operators.

\begin{theorem}
\label{t:2}
Let $m\geq 1$ be an integer, let 
$(\HH_i)_{0\leq i\leq m-1}$ be nonzero real Hilbert spaces,
set $\HH_m=\HH_0$, and let $\alpha \in [1/2,1]$. For every 
$i\in\{1,\ldots,m\}$, let $W_i\in\BL(\HH_{i-1},\HH_i)$ and let 
$P_i\colon\HH_i\to\HH_i$ be firmly nonexpansive. Suppose that 
$(W_i)_{1\leq i\leq m}$ satisfies Condition~\ref{o:1}. Then 
$P_m\circ W_m\circ\cdots\circ P_1\circ W_1$ is 
$\alpha$-averaged.
\end{theorem}
\begin{proof}
Set $T=P_m\circ W_m\circ\cdots\circ P_1\circ W_1$. 
We must show that
\begin{equation}
\label{e:fz84}
Q=\bigg(1-\frac{1}{\alpha}\bigg)\Id+\frac{1}{\alpha}T
\end{equation}
is nonexpansive. By assumption, for every $i\in\{1,\ldots,m\}$, 
there exists a nonexpansive operator $Q_i\colon\HH_i\to\HH_i$ 
such that \eqref{e:fz79} holds.
Let $(L_i)_{1\leq i\leq m}$ be as in \eqref{e:defLi} and let $x$ 
and $y$ be distinct points in $\HH_0$. According to 
Lemma~\ref{l:peniblegenb}, there exists 
$\boldsymbol{v}=(v_0,\ldots,v_{m-1})\in\HH_0\times
\cdots\times\HH_{m-1}$ such that
\begin{equation}
\begin{cases}
\displaystyle v_0=\frac{x-y}{\|x-y\|}\\
(\forall i\in\{1,\ldots,m-1\})\quad \|v_i\|
\leq\|L_i(v_0,\ldots,v_{i-1})\|\\[2mm]
\dfrac{2^{m-1}\big((W_m\circ P_{m-1}\circ\cdots\circ P_1\circ
W_1)x-(W_m\circ P_{m-1}\cdots\circ P_1\circ W_1)y\big)}
{\|x-y\|}=L_m\boldsymbol{v}.
\end{cases}
\end{equation}
Condition~\ref{o:1} imposes that
\begin{equation}
\|L_m\boldsymbol{v}-2^m(1-\alpha)v_0\|+\|L_m\boldsymbol{v}\|
\leq 2^m\alpha\|v_0\|=2^m\alpha,
\end{equation}
which is equivalent to 
\begin{multline}
\label{e:fz82}
\|(W_m\circ P_{m-1}\circ\cdots\circ P_1\circ W_1)x-
(W_m\circ P_{m-1}\cdots\circ P_1\circ W_1)y-2(1-\alpha)(x-y)\|\\
+\|(W_m\circ P_{m-1}\circ\cdots\circ P_1\circ W_1)x-
(W_m\circ  P_{m-1}\cdots\circ P_1\circ W_1)y\|\leq 2\alpha\|x-y\|.
\end{multline}
In turn, we derive from \eqref{e:fz84} and \eqref{e:fz79} that 
\begin{align}
&\hskip -6mm \|Qx-Qy\|\nonumber\\ 
&\leq\frac{1}{\alpha}
\Big\|\Big(\frac{\Id+Q_m}{2}\circ W_m\circ\cdots\circ P_1\circ
W_1\Big)x-\Big(\frac{\Id+Q_m}{2}\circ W_m\circ  \cdots\circ P_1\circ
W_1\Big)y-(1-\alpha)(x-y)\Big\|\nonumber\\
&\leq \frac{1}{2\alpha}
\Big(\|(W_m\circ P_{m-1}\circ\cdots\circ P_1\circ W_1)x-(W_m\circ
P_{m-1}\cdots\circ P_1\circ W_1)y-2(1-\alpha)(x-y)\|\nonumber\\
&\quad\;+\|(Q_m\circ W_m\circ P_{m-1}\circ\cdots\circ 
P_1\circ W_1)x-(Q_m\circ W_m\circ P_{m-1}\cdots\circ 
P_1\circ W_1)y\|\Big)\nonumber\\
&\leq \frac{1}{2\alpha}
\Big(\|(W_m\circ P_{m-1}\circ\cdots\circ P_1\circ W_1)x-(W_m\circ
P_{m-1}\cdots\circ P_1\circ W_1)y
-2(1-\alpha)(x-y)\|\nonumber\\
&\quad\;+\|(W_m\circ P_{m-1}\circ\cdots\circ P_1\circ W_1)x
-(W_m\circ P_{m-1}\cdots\circ P_1\circ W_1)y\|\Big)\nonumber\\
&\leq\|x-y\|,
\end{align}
which establishes the nonexpansiveness of $Q$.
\end{proof}

\begin{example}
Consider Theorem~\ref{t:2} with $m=2$. In view of 
Proposition~\ref{p:9}\ref{p:9iii}, 
$P_2\circ W_2\circ P_1\circ W_1$ is $\alpha$-averaged if 
$\|W_2\circ W_1-4(1-\alpha)\Id\|+\|W_2\circ W_1\|+2\|W_2\|\,\|W_1\|
\leq 4\alpha$. In particular, if $\alpha=1$, this condition is
obviously less restrictive than requiring that $W_1$ and $W_2$ be
nonexpansive.
\end{example}

\section{A variational inequality model}
\label{sec:4}

In this section, we first investigate an autonomous version of
Model~\ref{m:2}.

\begin{model}
\label{m:1}
This is the special case of Model~\ref{m:2} in which, for
every $i\in\{1,\ldots,m\}$, there exist 
$R_i\in\mathcal{A}(\HH_i)$, say $R_i=\prox_{\varphi_i}$ 
for some $\varphi_i\in\Gamma_0(\HH_i)$ with
$\varphi_i(0)=\inf\varphi_i(\HH_i)$, 
$W_i\in\BL(\HH_{i-1},\HH_i)$, and $b_i\in\HH_i$ such that
$(\forall n\in\NN)$ $R_{i,n}=R_i$, $W_{i,n}=W_i$,
$b_{i,n}=b_i$. We set 
\begin{equation}
\label{e:sj1}
(\forall i\in\{1,\ldots,m\})\quad
T_i\colon\HH_{i-1}\to\HH_i\colon x\mapsto R_i(W_ix+b_i)
\end{equation} 
and
\begin{equation}
\label{e:F}
\begin{cases}
F=\Fix(T_m\circ\cdots\circ T_1)\\
\HHH=\HH_1\oplus\cdots\oplus\HH_{m-1}\oplus\HH_m\\
\overset{\rightarrow}{\HHH}=
\HH_m\oplus\HH_1\oplus\cdots\oplus\HH_{m-1}\\
\boldsymbol{S}\colon\HHH\to\overset{\rightarrow}{\HHH}\colon
(x_1,\ldots,x_{m-1},x_m)\mapsto(x_m,x_1,\ldots,x_{m-1})\\
\boldsymbol{W}\colon\overset{\rightarrow}{\HHH}\to\HHH\colon
(x_m,x_1,\ldots,x_m)\mapsto(W_1x_m,W_2x_1,\ldots,W_mx_{m-1})\\
\boldsymbol{\varphi}\colon\HHH\to\RX\colon\boldsymbol{x}
\mapsto\sum_{i=1}^m\varphi_i(x_i)\\
\boldsymbol{\psi}\colon\HHH\to\RX\colon\boldsymbol{x}
\mapsto\sum_{i=1}^m\big(\varphi_i(x_i)-\scal{x_i}{b_i}\big)\\
\boldsymbol{F}=
\menge{\boldsymbol{x}\in\HHH}
{x_1=T_1x_m,\;x_2=T_2x_1,\ldots,\;x_m=T_mx_{m-1}},
\end{cases}
\end{equation}
where $\boldsymbol{x}=(x_1,\ldots,x_m)$ denotes a generic 
element in $\HHH$.
\end{model}

\subsection{Static analysis}

We start with a property of the compositions of the operators 
$(T_i)_{1\leq i\leq m}$ of \eqref{e:sj1}.

\begin{proposition}
\label{p:boundcompTib}
Consider the setting of Model~\ref{m:1}, let $i$ and $j$ be
integers such that $1\leq j\leq i\leq m$, and let $x\in\HH_{j-1}$. 
Then 
\begin{equation}
\label{e:fz69}
\|(T_i\circ\cdots\circ T_j)x\|\leq\|x\|
\prod_{k=j}^{i}\|W_k\|
+\sum_{q=j}^{i}\Bigg(\|b_q\|\prod_{k=q+1}^{i}\|W_k\|\Bigg).
\end{equation}
\end{proposition}
\begin{proof}
In view of \eqref{e:sj1}, the property is satisfied when $i=j$.
We now assume that $i>j$. Since $R_i\in\mathcal{A}(\HH_i)$,
Proposition~\ref{p:5}\ref{p:5i-} yields
\begin{align}
\|(T_i\circ\cdots\circ T_j)x\|
&=\|R_i(W_i(T_{i-1}\circ\cdots\circ T_j)x+b_i)\|\nonumber\\
&=\|R_i(W_i(T_{i-1}\circ\cdots\circ T_j)x+b_i)-R_i0\|
\nonumber\\
&\leq\|W_i(T_{i-1}\circ\cdots\circ T_j)x+b_i\|\nonumber\\
&\leq\|W_i\|\,\|(T_{i-1}\circ\cdots\circ T_j)x\|+\|b_i\|.
\end{align}
We thus obtain \eqref{e:fz69} recursively. 
\end{proof}

Next, we establish a connection between Model~\ref{m:1} and a
variational inequality. 
\begin{proposition}
\label{p:10}
In the setting of Model~\ref{m:1}, consider the variational
inequality problem
\begin{equation}
\label{e:vi1}
\text{find}\;\;\overline{x}_1\in\HH_1,\ldots,\,
\overline{x}_m\in\HH_m
\;\;\text{such that}\quad
\begin{cases}
b_1\in
\overline{x}_1-W_1\overline{x}_m+\partial
\varphi_1(\overline{x}_1)\\
b_2\in\overline{x}_2-W_2\overline{x}_1+
\partial\varphi_2(\overline{x}_2)\\
\hskip 6mm\vdots\\
b_m\in\overline{x}_m-W_m\overline{x}_{m-1}+
\partial\varphi_m(\overline{x}_m).
\end{cases}
\end{equation}
Then the following hold:
\begin{enumerate}
\itemsep0mm 
\item
\label{p:10i}
The set of solutions to \eqref{e:vi1} is $\boldsymbol{F}$.
\item
\label{p:10ii}
$\boldsymbol{F}=\zer(\ID-\boldsymbol{W}\circ\boldsymbol{S}+
\partial\boldsymbol{\psi})=
\Fix(\prox_{\boldsymbol{\psi}}\circ\boldsymbol{W}
\circ\boldsymbol{S})$.
\item
\label{p:10iv}
$\boldsymbol{F}=\menge{(T_1\overline{x}_m,(T_2\circ T_1)
\overline{x}_m,\ldots, (T_{m-1}\circ\cdots\circ T_1)
\overline{x}_m,\overline{x}_m)}{\overline{x}_m\in F}$.
\item
\label{p:10iii}
Suppose that $(W_i)_{1\leq i\leq m}$ satisfies 
Condition~\ref{o:1} for some $\alpha\in [1/2,1]$.
Then $F$ is closed and convex. 
\item
\label{p:10iii+}
Suppose that $(W_i)_{1\leq i \leq m}$
satisfies Condition~\ref{o:1} for some $\alpha\in [1/2,1]$
and that one of the following holds:
\begin{enumerate}[label=\rm(\alph*)]
\itemsep0mm 
\item 
\label{p:10iii+a}
$\ran(T_m\circ\cdots\circ T_1)$ is bounded.
\item 
\label{p:10iii+b}
There exists $j\in\{1,\ldots,m\}$ such that $\dom\varphi_j$ 
is bounded. 
\end{enumerate}
Then $F$ and $\boldsymbol{F}$ are nonempty.
\item
\label{p:10v}
Suppose that $\ID-\boldsymbol{W}\circ\boldsymbol{S}$ is monotone.
Then $\boldsymbol{F}$ is closed and convex. In addition, $F$ and
$\boldsymbol{F}$ are nonempty if any of the following holds:
\begin{enumerate}[label=\rm(\alph*)]
\itemsep0mm 
\item
\label{p:10via}
$\ID-\boldsymbol{W}\circ\boldsymbol{S}+
\partial\boldsymbol{\varphi}$ is surjective.
\item
\label{p:10vib}
$\partial\boldsymbol{\varphi}-\boldsymbol{W}\circ\boldsymbol{S}$ 
is maximally monotone.
\item
\label{p:10vic}
$\max_{1\leq i\leq m}\|W_i\|\leq 1$,
$\boldsymbol{S}^*-\boldsymbol{W}$ has closed range, and 
$\ker(\boldsymbol{S}-\boldsymbol{W}^*)=\{\boldsymbol{0}\}$.
\item
\label{p:10vid}
$\max_{1\leq i\leq m}\|W_i\|\leq 1$ and, for every 
$i\in\{1,\ldots,m\}$, $\dom\varphi_i^*=\HH_i$.
\item
\label{p:10vie}
For every $i\in\{1,\ldots,m\}$,
$\dom\varphi_i=\HH$ and $\dom\varphi_i^*=\HH_i$.
\item
\label{p:10vig}
$\boldsymbol{S}^*-\boldsymbol{W}$ has closed range, 
$\ker(\boldsymbol{S}-\boldsymbol{W}^*)=\{\boldsymbol{0}\}$,
and, for every $i\in\{1,\ldots,m\}$, $\dom\varphi_i=\HH_i$.
\item
\label{p:10vif}
For every $i\in\{1,\ldots,m\}$, $\dom\varphi_i$ is bounded.
\end{enumerate}
\end{enumerate}
\end{proposition}
\begin{proof}
We first observe that 
$\boldsymbol{S}\in\BL(\HHH,\overset{\rightarrow}{\HHH})$,
$\boldsymbol{W}\in\BL(\overset{\rightarrow}{\HHH},\HHH)$,  
$\boldsymbol{\varphi}\in\Gamma_0(\HHH)$, and
$\boldsymbol{\psi}\in\Gamma_0(\HHH)$.

\ref{p:10i}: Let $\boldsymbol{x}\in\HHH$. Then
\begin{eqnarray}
\label{e:via2}
\boldsymbol{x}\;\text{solves \eqref{e:vi1}}
&\Leftrightarrow&
\begin{cases}
W_1x_m+b_1\in x_1+\partial\varphi_1(x_1)\\
W_2x_1+b_2\in x_2+\partial\varphi_2(x_2)\\
\hskip 6mm\vdots\\
W_mx_{m-1}+b_m\in x_m+\partial\varphi_m(x_m).
\end{cases}\\
&\Leftrightarrow&
\begin{cases}
x_1=\prox_{\varphi_1}(W_1x_m+b_1)=T_1x_m\\
x_2=\prox_{\varphi_2}(W_2x_1+b_2)=T_2x_1\\
\hskip 6mm\vdots\\
x_m=\prox_{\varphi_m}(W_mx_{m-1}+b_m)=T_mx_{m-1}.
\end{cases}
\end{eqnarray}

\ref{p:10ii}: Let $\boldsymbol{x}\in\HHH$. Using \eqref{e:F}, 
we obtain 
\begin{equation}
\label{e:via1}
\boldsymbol{x}\;\text{solves \eqref{e:vi1}}
\quad\Leftrightarrow\quad
\boldsymbol{0}\in\boldsymbol{x}-
\boldsymbol{W}(\boldsymbol{S}\boldsymbol{x})+
\partial\boldsymbol{\psi}(\boldsymbol{x})
\quad\Leftrightarrow\quad
\boldsymbol{x}=\prox_{\boldsymbol{\psi}}
\big(\boldsymbol{W}(\boldsymbol{S}\boldsymbol{x})\big).
\end{equation}

\ref{p:10iv}: Clear from the definitions of $F$ and
$\boldsymbol{F}$.

\ref{p:10iii}: Define $m$ firmly nonexpansive operators by
$(\forall i\in\{1,\ldots,m\})$ 
$P_i\colon\HH_i\to\HH_i\colon y\mapsto R_i(y+b_i)$.
Then it follows from \eqref{e:sj1} and Theorem~\ref{t:2} applied 
to $(P_i)_{1\leq i\leq m}$ that $T_m\circ\cdots\circ T_1$ is 
nonexpansive. In turn, we derive from \cite[Corollary~4.24]{Livre1} 
that its fixed point set $F$ is closed and convex. 

\ref{p:10iii+}: Thanks to \ref{p:10iv}, it is enough to show that
$F\neq\emp$. Set $T=T_m\circ\cdots\circ T_1$ and recall that
it is nonexpansive by virtue of Theorem~\ref{t:2}.

\ref{p:10iii+a}: 
Let $C$ be a closed ball such that $\ran T\subset C$ and set
$S=T|_C$. Then $S\colon C\to C$ is nonexpansive and therefore
\cite[Proposition~4.29]{Livre1} asserts that 
$\Fix T=\Fix S\neq\emp$.

\ref{p:10iii+b}$\Rightarrow$\ref{p:10iii+a}:
We have $\ran T_j\subset\ran R_j
=\ran\prox_{\varphi_j}
=\dom(\Id+\partial\varphi_j)
=\dom\partial\varphi_j\subset\dom\varphi_{j}$. Hence $\ran T_j$ is
bounded and Proposition~\ref{p:boundcompTib} 
(with $i=m$) implies that 
\begin{equation}
\ran T \subset 
\begin{cases}
\ran T_m, & \text{if}\;\;j=m;\\
(T_m\circ\cdots\circ T_{j+1})(\ran T_j), 
& \text{if}\;\;1\leq j\leq m-1
\end{cases}
\end{equation}
is likewise. 

\ref{p:10v}: Set $\boldsymbol{A}=\ID-\boldsymbol{W}\circ
\boldsymbol{S}+\partial\boldsymbol{\psi}$. Since
$\ID-\boldsymbol{W}\circ\boldsymbol{S}$ is 
monotone and continuous, it is maximally monotone
\cite[Corollary~20.28]{Livre1}, with 
$\HHH$ as its domain. Since $\partial\boldsymbol{\psi}$ is also
maximally monotone \cite[Theorem~20.25]{Livre1}, $\boldsymbol{A}$ 
is likewise \cite[Corollary~25.5(i)]{Livre1} 
and hence $\boldsymbol{F}=\zer\boldsymbol{A}$ is closed and convex
\cite[Proposition~23.39]{Livre1}.
Next, we note that, in view of \ref{p:10iv}, $F\neq\emp$ 
$\Leftrightarrow$ $\boldsymbol{F}\neq\emp$.

\ref{p:10via}: The hypothesis implies that
$(b_i)_{1\leq i\leq m}\in\ran(\ID-\boldsymbol{W}\circ
\boldsymbol{S}+\partial\boldsymbol{\varphi})$ and therefore 
that \eqref{e:vi1} has a solution, i.e., $\boldsymbol{F}\neq\emp$.

\ref{p:10vib}$\Rightarrow$\ref{p:10via}: 
The claim follows from Minty's theorem \cite[Theorem~21.1]{Livre1}.

\ref{p:10vic}$\Rightarrow$\ref{p:10via}: 
We have $\|\boldsymbol{W}\circ\boldsymbol{S}\|=\|\boldsymbol{W}\|
=\max_{1\leq i\leq m}\|W_i\|\leq 1$. Therefore,
$-\boldsymbol{W}\circ\boldsymbol{S}$ is nonexpansive, which 
implies that
$(\ID-\boldsymbol{W}\circ\boldsymbol{S})/2$ is firmly
nonexpansive \cite[Corollary~4.5]{Livre1}, that is 
$(\forall\boldsymbol{x}\in\HHH)$
$\scal{\boldsymbol{x}-\boldsymbol{W}(\boldsymbol{S}\boldsymbol{x})}
{\boldsymbol{x}}\geq\|\boldsymbol{x}-
\boldsymbol{W}(\boldsymbol{S}\boldsymbol{x})\|^2/2$. Consequently,
$\ID-\boldsymbol{W}\circ\boldsymbol{S}$ is $3^*$~monotone
\cite[Proposition~25.16]{Livre1}, while
$\partial\boldsymbol{\varphi}$ is also $3^*$~monotone 
\cite[Example~25.13]{Livre1}.
Finally, since $\boldsymbol{S}$ is unitary,
\begin{equation}
\ran\big(\ID-\boldsymbol{W}\circ\boldsymbol{S}\big)
=\ran\big(\boldsymbol{S}^*-\boldsymbol{W}\big)
=\cran\big(\boldsymbol{S}-\boldsymbol{W}^*\big)^*
=\Big(\ker\big(\boldsymbol{S}-\boldsymbol{W}^*\big)\Big)^\bot
=\HHH,
\end{equation}
which shows that
$\ID-\boldsymbol{W}\circ\boldsymbol{S}$ is surjective.
Altogether, since \cite[Corollary~25.5(i)]{Livre1} implies that
$\ID-\boldsymbol{W}\circ\boldsymbol{S}
+\partial\boldsymbol{\varphi}$ is maximally monotone, it 
follows from \cite[Corollary~25.27(i)]{Livre1}
that $\ID-\boldsymbol{W}\circ\boldsymbol{S}+
\partial\boldsymbol{\varphi}$ is surjective.

\ref{p:10vid}$\Rightarrow$\ref{p:10via}: 
We have $\dom\boldsymbol{\varphi}^*=\HHH$. Hence
since $\intdom\boldsymbol{\varphi}^*
\subset\dom\partial\boldsymbol{\varphi}^*$
\cite[Proposition~16.27]{Livre1}, we have 
$\ran\partial\boldsymbol{\varphi}
=\dom(\partial\boldsymbol{\varphi})^{-1}
=\dom\partial\boldsymbol{\varphi}^*=\HHH$.
Hence, $\partial\boldsymbol{\varphi}$ is surjective.  
We conclude using the same arguments as in \ref{p:10vic}:
$\partial\boldsymbol{\varphi}$ and 
$\ID-\boldsymbol{W}\circ\boldsymbol{S}$
are both $3^*$~monotone and their
sum is maximally monotone, which allows us to invoke 
\cite[Corollary~25.27(i)]{Livre1}.

\ref{p:10vie}$\Rightarrow$\ref{p:10via}: 
As seen in \ref{p:10vid}, 
$\partial\boldsymbol{\varphi}$ is surjective.
We have $\HHH=\intdom\boldsymbol{\varphi}\subset
\dom\partial\boldsymbol{\varphi}$
\cite[Proposition~16.27]{Livre1}. Consequently, 
$\HHH=\dom(\ID-\boldsymbol{W}\circ\boldsymbol{S})\subset
\dom\partial\boldsymbol{\varphi}$. Altogether, since 
$\partial\boldsymbol{\varphi}$ is $3^*$~monotone, it follows
from \cite[Corollary~25.27(ii)]{Livre1} that 
$\ID-\boldsymbol{W}\circ\boldsymbol{S}+
\partial\boldsymbol{\varphi}$ is surjective.

\ref{p:10vig}$\Rightarrow$\ref{p:10via}: 
As seen in \ref{p:10vic}, $\ID-\boldsymbol{W}\circ\boldsymbol{S}$
is surjective and $\partial\boldsymbol{\varphi}$ 
is $3^*$~monotone. In addition, 
$\dom(\ID-\boldsymbol{W}\circ\boldsymbol{S})
\subset\dom\partial\boldsymbol{\varphi}$ since
$\HHH=\intdom\boldsymbol{\varphi}\subset
\dom\partial\boldsymbol{\varphi}$ \cite[Proposition~16.27]{Livre1}. 
Altogether, it follows from \cite[Corollary~25.27(ii)]{Livre1} 
that $\ID-\boldsymbol{W}\circ\boldsymbol{S}+
\partial\boldsymbol{\varphi}$ is surjective.

\ref{p:10vif}: Here $\dom\boldsymbol{A}=\dom\partial
\boldsymbol{\varphi}\subset\dom\boldsymbol{\varphi}=
\cart_{\!i=1}^{\!m}\dom\varphi_i$ is bounded. Hence, 
$\boldsymbol{F}=\zer\boldsymbol{A}\neq\emp$ 
\cite[Proposition~23.36(iii)]{Livre1}.
\end{proof}

\begin{remark}
In Proposition~\ref{p:10}\ref{p:10v}, it is required that 
$\ID-\boldsymbol{W}\circ\boldsymbol{S}$ be monotone, or
equivalently, that its self-adjoint part 
$\ID-(\boldsymbol{W}\circ\boldsymbol{S}+
\boldsymbol{S}^*\circ\boldsymbol{W}^*)/2$ be positive. 
In a finite-dimensional setting, this just means that the 
eigenvalues of the matrix $\boldsymbol{W}
\boldsymbol{S}+\boldsymbol{S}^*\boldsymbol{W}^*$ are 
in $\left]\minf,2\right]$.
\end{remark}

\begin{remark}
\label{r:93}
Let $\overline{\boldsymbol{x}}\in\HHH$ be a solution to the 
variational inequality \eqref{e:vi1}. A natural question is whether 
$\overline{\boldsymbol{x}}$ solves a minimization problem. 
In general the answer
is negative. For instance, for $m\geq 3$ layers, even if the
Hilbert spaces $(\HH_i)_{1\leq i\leq m}$ are identical,
$\boldsymbol{W}=\ID$, the vectors $(b_i)_{1\leq i\leq m}$ are
zero, and the functions $(\varphi_i)_{1\leq i\leq m}$ are indicator
functions of closed convex sets $(C_i)_{1\leq i\leq m}$, the
solutions to \eqref{e:vi1} do not minimize any function
$\boldsymbol{\Phi}\colon\HHH\to\RR$ \cite{Jfan12}. A rather
restrictive scenario in which the answer is positive is when
$\ID-\boldsymbol{W}\circ\boldsymbol{S}$ is monotone and 
$\boldsymbol{W}\circ\boldsymbol{S}$ is self-adjoint. Then 
$\overline{\boldsymbol{x}}$ is a minimizer of 
$\boldsymbol{\Phi}\colon\boldsymbol{x}\mapsto(1/2)
\scal{\boldsymbol{x}-\boldsymbol{W}(\boldsymbol{S}\boldsymbol{x})}
{\boldsymbol{x}}+\boldsymbol{\psi}(\boldsymbol{x})$.
\end{remark}

\begin{example}
In Model~\ref{m:1}, suppose that, for every $i\in\{1,\ldots,m\}$,
$\HH_i=\RR^{N_i}$ for some strictly positive integer $N_i$.
In addition, assume that, for every $i\in\{1,\ldots,m\}$, $R_i$ is
a separable activation operator with respect to the canonical 
basis of $\RR^{N_i}$ (see Proposition~\ref{p:4}), and that it
employs the ReLU activation functions of Example~\ref{ex:1}. 
For every $i\in\{1,\ldots,m\}$, let 
$x_i=(\xi_{i,k})_{1\leq k\leq N_i}\in\RR^{N_i}$ and set 
$b_i=(\beta_{i,k})_{1\leq k\leq N_i}$. Then it follows from
Proposition~\ref{p:10}\ref{p:10i}
that $(x_1,\ldots,x_m)\in\boldsymbol{F}$ if and only if, 
for every $i\in\{1,\ldots,m\}$, $x_i\in\RP^{N_i}$ and
\begin{multline}
\begin{cases}
(\forall k\in\{1,\ldots,N_1\})\quad 
[W_1x_m]_k+\beta_{1,k}-\xi_{1,k}\in\mathcal{I}(\xi_{1,k})
\\
(\forall k\in\{1,\ldots,N_2\})\quad 
[W_2x_1]_k+\beta_{2,k}-\xi_{2,k}\in\mathcal{I}(\xi_{2,k})
\\
\qquad\vdots\\
(\forall k\in\{1,\ldots,N_{m-1}\})\quad 
[W_{m-1}
x_{m-2}]_k+\beta_{m-1,k}-\xi_{m-1,k}\in\mathcal{I}(\xi_{m-1,k})
\\
(\forall k\in\{1,\ldots,N_m\})\quad 
[W_m x_{m-1}]_k+\beta_{m,k}-\xi_{m,k}\in\mathcal{I}(\xi_{m,k}),
\end{cases}
\end{multline}
where, given $x\in\HH_{i-1}$, $[W_ix]_k$ is the $k$th component of 
$W_ix$ and
\begin{equation}
(\forall\xi\in\RP)\quad 
\mathcal{I}(\xi)=
\begin{cases}
\{0\}, &\text{if}\;\;\xi\in\RPP;\\
\RM, &\text{if}\;\;\xi=0.
\end{cases}
\end{equation}
Altogether, we conclude that $\boldsymbol{F}$ is a closed convex
polyhedron.
\end{example}

\subsection{Asymptotic analysis}

Next, we investigate the asymptotic behavior of \eqref{e:algo4} in
the context of Model \ref{m:1}.

\begin{theorem}
\label{t:1}
In the setting of Model~\ref{m:1}, set $T=T_m\circ\cdots\circ T_1$, 
let $\alpha\in [1/2,1]$, and suppose that the following hold:
\begin{enumerate}[label=\rm(\alph*)]
\itemsep0mm 
\item
\label{t:1i}
$F\neq\emp$.
\item
\label{t:1ii}
$(W_i)_{1\leq i\leq m}$ satisfies Condition~\ref{o:1} with
parameter $\alpha$.
\item
\label{t:1iii}
One of the following is satisfied:
\begin{enumerate}[label=\rm(\roman*)]
\itemsep0mm 
\item
\label{t:1ci}
$\lambda_n\equiv 1/\alpha=1$ and $Tx_n-x_n\to 0$.
\item
\label{t:1cii}
$(\lambda_n)_{n\in\NN}$ lies in $\left]0,1/\alpha\right[$ and
$\sum_{n\in\NN}\lambda_n(1-\alpha\lambda_n)=\pinf$.
\end{enumerate}
\end{enumerate}
Then $(x_n)_{n\in\NN}$ converges weakly to a point 
$\overline{x}_m\in F$ and $(T_1\overline{x}_m,(T_2\circ T_1)
\overline{x}_m,\ldots, (T_{m-1}\circ\cdots\circ T_1)
\overline{x}_m,\overline{x}_m)$
solves \eqref{e:vi1}. Now suppose that, in addition, any of 
the following holds:
\begin{enumerate}
\itemsep0mm 
\setcounter{enumi}{2}
\item
\label{t:1d}
For every $i\in\{1,\ldots,m-1\}$, $R_i$ is weakly sequentially 
continuous.
\item
\label{t:1e}
For every $i\in\{1,\ldots,m-1\}$, $R_i$ is a separable activation
operator in the sense of Proposition~\ref{p:4}.
\item
\label{t:1f}
For every $i\in\{1,\ldots,m-1\}$, $\HH_i$ is finite-dimensional.
\item
\label{t:1g}
For some $\varepsilon\in\left]0,1/2\right[$, 
$(\lambda_n)_{n\in\NN}$ lies in  
$[\varepsilon,(1-\varepsilon)(\varepsilon+1/\alpha)]$ and, 
for every $i\in\{1,\ldots,m\}$, $\HH_i=\HH$ and there exists 
$\beta_i\in\zeroun$ such that
$\|W_i-2(1-\beta_i)\Id\|+\|W_i\|\leq 2\beta_i$.
\end{enumerate}
Then, for every $i\in\{1,\ldots,m-1\}$, $(x_{i,n})_{n\in\NN}$ 
converges weakly to 
$\overline{x}_i=(T_i\circ\cdots\circ T_1)\overline{x}_m$ 
and $(\overline{x}_1,\ldots,\overline{x}_m)$ solves \eqref{e:vi1}.
\end{theorem}
\begin{proof}
We first derive from \eqref{e:algo4} and Model~\ref{m:1} that
\begin{equation}
\label{e:i5}
(\forall n\in\NN)\quad x_{n+1}=x_n+\lambda_n(Tx_n-x_n).
\end{equation}
Now set $(\forall i\in\{1,\ldots,m\})$ 
$P_i\colon\HH_i\to\HH_i\colon y\mapsto R_i(y+b_i)$.
Then \eqref{e:sj1} yields
$T=P_m\circ W_m\circ\cdots\circ P_1\circ W_1$ and, since the 
operators $(R_i)_{1\leq i\leq m}$ are firmly nonexpansive, the 
operators $(P_i)_{1\leq i\leq m}$ are likewise. Hence, it follows 
from \ref{t:1ii}, Theorem~\ref{t:2}, and \eqref{e:F} that
\begin{equation}
\label{e:a124}
T\;\text{is $\alpha$-averaged and}\;\Fix T=F.
\end{equation}

\ref{t:1ci}: In view of \eqref{e:a124}, $T$ is nonexpansive and
hence we derive from \cite[Theorem~5.14(i)]{Livre1} that
$(x_n)_{n\in\NN}$ converges weakly to a point in $F$. The second 
assertion follows from Proposition~\ref{p:10}\ref{p:10iv}.

\ref{t:1cii}: In view of \eqref{e:a124}, $T$ is $\alpha$-averaged
with $\alpha<1$. In turn, \cite[Proposition~5.16(iii)]{Livre1}
implies that $(x_n)_{n\in\NN}$ converges weakly to a point in $F$,
and we conclude by invoking Proposition~\ref{p:10}\ref{p:10iv}.

We now prove the convergence of the individual
sequences under each assumption.

\ref{t:1d}: We have already established that 
$x_n\weakly\overline{x}_m$. Since $W_1$ is
weakly continuous as a bounded linear operator, so is $T_1$ in
\eqref{e:sj1}. Hence, \eqref{e:algo4} implies that
$x_{1,n}=T_1x_n\weakly T_1\overline{x}_m=\overline{x}_1$. Likewise,
we obtain successively
$x_{2,n}=T_2x_{1,n}\weakly T_2\overline{x}_1=\overline{x}_2$,
$x_{3,n}=T_3x_{2,n}\weakly T_3\overline{x}_2=\overline{x}_3$,\ldots,
$x_{m,n}=T_mx_{m-1,n}\weakly T_m\overline{x}_{m-1}=\overline{x}_m$.

\ref{t:1e}$\Rightarrow$\ref{t:1d}: 
See \cite[Proposition~24.12(iii)]{Livre1}.

\ref{t:1f}$\Rightarrow$\ref{t:1d}: A proximity operator 
is nonexpansive and therefore continuous, hence weakly continuous
in a finite-dimensional setting.

\ref{t:1g}: As shown above, $x_n\weakly\overline{x}_m\in F$. 
It follows from Proposition~\ref{p:9}\ref{p:9iii}
and Theorem~\ref{t:2} (applied with $m=1$) that, for every
$i\in\{1,\ldots,m\}$, $T_i$ is $\beta_i$-averaged. Hence, upon
applying \cite[Theorem~3.5(ii)]{Jmaa15} with $\alpha$ as an 
averaging constant of $T$, we infer that 
\begin{equation}
\label{e:kj2018}
\begin{cases}
(\Id-T_1)x_n-(\Id-T_1)\overline{x}_m\to 0\\
(\Id-T_2)(T_1x_n)-(\Id-T_2)(T_1\overline{x}_m)\to 0\\
\hskip 12mm\vdots\\
(\Id-T_m)((T_{m-1}\circ\cdots\circ T_1)x_n)-
(\Id-T_m)((T_{m-1}\circ\cdots\circ T_1)\overline{x}_m)\to 0.
\end{cases}
\end{equation}
Thus, 
$x_{1,n}-x_n=T_1x_n-x_n\to T_1\overline{x}_m-\overline{x}_m$, which
implies that $x_{1,n}=(x_{1,n}-x_n)+x_n\weakly 
(T_1\overline{x}_m-\overline{x}_m)+\overline{x}_m=
T_1\overline{x}_m$. However, since 
$x_{2,n}-x_{1,n}=(T_2\circ T_1)x_n-T_1x_n\to
(T_2\circ T_1)\overline{x}_m-T_1\overline{x}_m$, we obtain 
$x_{2,n}\weakly (T_2\circ T_1)\overline{x}_m$. Continuing this
telescoping process yields the claim. 
\end{proof}

The next result covers the case when the variational inequality
problem \eqref{e:vi1} has no solution.

\begin{proposition}
\label{p:20}
In the setting of Model~\ref{m:1}, suppose that 
$(W_i)_{1\leq i\leq m}$ satisfies Condition~\ref{o:1} 
with $\alpha\in [1/2,1]$, that 
$(\lambda_n)_{n\in\NN}$ lies in
$[\varepsilon,(1/\alpha)-\varepsilon]$, for some 
$\varepsilon\in\left]0,1/2\right[$, and that
$F=\emp$. Then $\|x_n\|\to\pinf$.
\end{proposition}
\begin{proof}
We derive from \eqref{e:i5} and \eqref{e:a124} that, for every
$n\in\NN$, $x_{n+1}=x_n+\mu_n(Qx_n-x_n)$, where 
$Q=(1-1/\alpha)\Id+(1/\alpha)T$ is nonexpansive and 
such that $\Fix Q=F$, and $\mu_n=\alpha\lambda_n\in\zeroun$. Hence
the claims follows from \cite[Proposition~4.29]{Livre1} and
\cite[Corollary~9(b)]{Borw92}.
\end{proof}

\begin{remark}
When assumptions~\ref{t:1i}--\ref{t:1iii} in Theorem~\ref{t:1} are
satisfied, the neural network described in Model~\ref{m:2} is
robust to perturbations of its input. Indeed, since $T$ is
$\alpha$-averaged in \eqref{e:i5}, we can write the updating rule
as $x_{n+1}=Q_nx_n$, where $Q_n$ is nonexpansive. In turn, 
if $x_0$ and $\widetilde{x}_0$ are 
two inputs in $\HH_0$, for a given $n\in\NN$, the resulting 
outputs $x_n$ and $\widetilde{x}_n$
are such that $\|x_n-\widetilde{x}_n\|\leq\|x_0-\widetilde{x}_0\|$.
\end{remark}

\begin{remark}
In connection with Theorem~\ref{t:1} and Remark~\ref{r:93}, 
let us underline that in general the weak limit $\overline{x}_m$ 
of $(x_n)_{n\in\NN}$ does not solve a 
minimization problem. A very special case in which it does
is the following. Suppose that $m=2$, $\HH_1=\HH$, 
$\|W_1\|\leq 1$, and $W_2=W_1^*$. 
Set $\psi_1=\varphi_1-\scal{\cdot}{b_1}$ and 
$\psi_2=\varphi_2-\scal{\cdot}{b_2}$, and let 
$\overline{x}_2\in F$, i.e., $\overline{x}_2=(\prox_{\psi_2}\circ 
W_1^*\circ\prox_{\psi_1}\circ W_1)\overline{x}_2$. It follows from
\cite[Remark~3.10(iv)]{Comb18} that there exists a function
$\vartheta\in\Gamma_{0}(\HH)$ such that $W_{1}^*\circ
\prox_{\psi_1}\circ W_1=\prox_{\vartheta}$. Thus,
$\overline{x}_2$ is a fixed point of the backward-backward
operator $\prox_{\psi_2}\circ\prox_{\vartheta}$. It then follows
from \cite[Remark~6.13]{Opti04} that $\overline{x}_2$
is a minimizer of $\moyo{\vartheta}{1}+\psi_2$, where 
$\moyo{\vartheta}{1}\colon x\mapsto\inf_{y\in\HH}
(\vartheta(y)+\|x-y\|^2/2)$ is the Moreau envelope of $\vartheta$.
\end{remark}

\begin{remark}
To model closely existing deep neural networks, we have chosen the
activation operators in Definition~\ref{d:2} and Model~\ref{m:1} 
to be proximity operators. However, as is clear from the results of
Section~\ref{sec:2} and in particular the central Theorem~\ref{t:2},
an activation operator $R_i\colon\HH_i\to\HH_i$ could more 
generally be a firmly nonexpansive operator that admits $0$ as a 
fixed point. By \cite[Corollary~23.9]{Livre1}, this means that 
$R_i$ is the resolvent of some maximally monotone operator such
$A_i\colon\HH_i\to 2^{\HH_i}$ (i.e., $R_i=(\Id+A_i)^{-1}$)
such that $0\in A_i0$. In this context, the variational inequality
\eqref{e:vi1} assumes the more general form of a system of
monotone inclusions, namely, 
\begin{equation}
\label{e:vi2}
\text{find}\;\;\overline{x}_1\in\HH_1,\ldots,\,
\overline{x}_m\in\HH_m
\;\;\text{such that}\quad
\begin{cases}
b_1\in
\overline{x}_1-W_1\overline{x}_m+A_1\overline{x}_1\\
b_2\in\overline{x}_2-W_2\overline{x}_1+A_2\overline{x}_2\\
\hskip 6mm\vdots\\
b_m\in\overline{x}_m-W_m\overline{x}_{m-1}+A_m 
\overline{x}_m.
\end{cases}
\end{equation}
\end{remark}

\section{Analysis of nonperiodic networks} 
\label{sec:5}

We analyze the deep neural network described in Model~\ref{m:2}
in the following scenario. 

\begin{assumption}
\label{a:main} 
In the setting of Model~\ref{m:2}, there exist sequences 
$(\omega_n)_{n\in\NN}\in\ell_+^1$,
$(\rho_n)_{n\in\NN}\in\ell_{+}^{1}$,
$(\eta_n)_{n\in\NN}\in\ell_{+}^{1}$, and
$(\nu_n)_{n\in\NN}\in\ell_+^1$ for which the following hold
for every $i\in\{1,\ldots,m\}$:
\begin{enumerate}
\itemsep0mm 
\item 
\label{a:maini} 
There exists
$W_i\in\BL(\HH_{i-1},\HH_i)$  such that $(\forall n \in \NN)$
$\|W_{i,n}-W_i\|\leq\omega_n$.
\item 
\label{a:mainii} 
There exists
$R_i\in\mathcal{A}(\HH_i)$ such that 
$(\forall n\in\NN)(\forall x\in\HH_i)$ 
$\|R_{i,n}x-R_ix\|\leq\rho_n\|x\|+\eta_n$.
\item
\label{a:mainiii} 
There exists $b_i\in\HH_i$ such that 
$(\forall n\in\NN)$ $\|b_{i,n}-b_i\|\leq\nu_n$.
\end{enumerate}
In addition, we set
\begin{equation}
\label{e:sj10}
(\forall i\in\{1,\ldots,m\})\quad
T_i\colon\HH_{i-1}\to\HH_i\colon x\mapsto R_i(W_ix+b_i).
\end{equation} 
\end{assumption}

\begin{proposition}
\label{p:difTin}
In the setting of Model~\ref{m:2}, suppose that 
Assumption~\ref{a:main} is satisfied, let $i\in\{1,\ldots,m\}$, 
and set
\begin{equation}
\label{e:defchizeta}
(\forall n\in\NN)\quad
\chi_{i,n}=\rho_n\|W_{i,n}\|+\omega_n\quad\text{and}\quad
\zeta_{i,n}=\rho_n\|b_{i,n}\|+\eta_n+\nu_n.
\end{equation}
Then $(\chi_{i,n})_{n\in\NN}\in\ell_{+}^{1}$,
$(\zeta_{i,n})_{n\in\NN}\in\ell_{+}^{1}$, and
$(\forall n\in\NN)(\forall x\in\HH_{i-1})$
$\|T_{i,n}x-T_ix\|\leq\chi_{i,n}\|x\|+\zeta_{i,n}$.
\end{proposition}
\begin{proof}
According to Assumptions~\ref{a:main}\ref{a:maini} and 
\ref{a:main}\ref{a:mainiii},
$\sup_{n\in\NN}\|W_{i,n}\|<\pinf$ and $\sup_{n\in\NN}
\|b_{i,n}\|<\pinf$. It then follows from \eqref{e:defchizeta}
that $(\chi_{i,n})_{n\in\NN}\in\ell_{+}^{1}$ and
$(\zeta_{i,n})_{n\in\NN}\in\ell_{+}^{1}$. 
Hence, we deduce from \eqref{e:sj4}, \eqref{e:sj10}, the
nonexpansiveness of $R_i$, and Assumption~\ref{a:main} that
\begin{align}
\label{e:difTi2}
&\hskip -13mm 
(\forall n\in\NN)(\forall x\in\HH_{i-1})\quad
\|T_{i,n}x-T_ix\| \nonumber\\
\qquad&\leq \|R_{i,n}(W_{i,n}x+b_{i,n})-R_i(W_{i,n}x+b_{i,n})\|
+\|R_i(W_{i,n}x+b_{i,n})-R_i(W_ix+b_i)\|\nonumber\\
&\leq \rho_n\|W_{i,n}x+b_{i,n}\|+\eta_n
+\|W_{i,n}x+b_{i,n}-W_ix-b_i\|\nonumber\\
&\leq \rho_n(\|W_{i,n}\|\,\|x\|+\|b_{i,n}\|)+\eta_n
+\|W_{i,n}-W_i\|\,\|x\|+\|b_{i,n}-b_i\|\nonumber\\
&\leq\rho_n(\|W_{i,n}\|\,\|x\|+\|b_{i,n}\|)+\eta_n
+\omega_n\,\|x\|+\nu_n\nonumber\\
&=\chi_{i,n}\|x\|+\zeta_{i,n},
\end{align}
as claimed.
\end{proof}

\begin{proposition}
\label{p:compTinoTi}
In the setting of Model~\ref{m:2}, suppose that 
Assumption~\ref{a:main} is satisfied.
Then, for every $i\in\{1,\ldots,m\}$, 
there exist $(\tau_{i,n})_{n\in\NN}\in\ell_{+}^{1}$ and
$(\theta_{i,n})_{n\in\NN}\in\ell_{+}^{1}$ such that
\begin{equation}
\label{e:fz90}
(\forall n \in\NN)(\forall x\in\HH)\quad
\|(T_{i,n}\circ\cdots\circ T_{1,n}) x-
(T_i\circ\cdots\circ T_1)x\|\leq\tau_{i,n}\|x\|+\theta_{i,n}.
\end{equation}
\end{proposition}
\begin{proof} 
For every $i\in\{1,\ldots,m\}$, define $(\chi_{i,n})_{n\in\NN}$
and $(\zeta_{i,n})_{n\in\NN}$ as in \eqref{e:defchizeta},
According to Proposition~\ref{p:difTin}, \eqref{e:fz90} is 
satisfied for $i=1$ by setting $(\forall n\in \NN)$ 
$\tau_{1,n}=\chi_{1,n}$ and 
$\theta_{1,n}=\zeta_{1,n}$. Next, let us assume that 
\eqref{e:fz90} holds for $i\in\{1,\ldots,m-1\}$ and set
\begin{equation}
(\forall n\in\NN)\quad
\begin{cases}
\tau_{i+1,n}=(\|W_{i+1}\|+\chi_{i+1,n})
\tau_{i,n}+\chi_{i+1,n}\Prod_{k=1}^{i}\|W_k\|\\
\theta_{i+1,n}=(\|W_{i+1}\|+\chi_{i+1,n})\theta_{i,n}+\chi_{i+1,n}
\Sum_{j=1}^i\Bigg(\|b_j\|\Prod_{k=j+1}^i\|W_k\|\Bigg)
+\zeta_{i+1,n}.
\end{cases}
\end{equation}
Then the sequences $(\tau_{i+1,n})_{n\in\NN}$
and $(\theta_{i+1,n})_{n\in\NN}$
belong to $\ell_{+}^{1}$.
Now let $n\in\NN$ and $x\in\HH$. Upon invoking 
Proposition~\ref{p:difTin}, the nonexpansiveness of $R_{i+1}$,
and Proposition~\ref{p:boundcompTib}, we obtain
\begin{align}
& \hskip -7mm  
\|(T_{i+1,n}\circ\cdots\circ T_{1,n}) x-(T_{i+1}\circ\cdots
\circ T_1)x\|\nonumber\\
&\hskip -4mm\leq\|(T_{i+1,n}\circ T_{i,n}\circ\cdots\circ T_{1,n})
x-(T_{i+1}\circ T_{i,n}\circ\cdots\circ T_{1,n})x\|\nonumber\\
&+\|(T_{i+1}\circ T_{i,n}\circ\cdots\circ T_{1,n})x- (T_{i+1}\circ 
T_i\circ\cdots\circ T_1)x\|\nonumber\\
&\hskip -4mm\leq\chi_{i+1,n}\|(T_{i,n}\circ\cdots\circ T_{1,n})x\|+
\zeta_{i+1,n}+
\|(T_{i+1}\circ T_{i,n}\circ\cdots\circ T_{1,n})x-(T_{i+1}\circ
T_i\circ\cdots\circ T_1)x\|\nonumber\\
&\hskip -4mm\leq \chi_{i+1,n} (\|(T_{i,n}\circ\cdots\circ T_{1,n})
x-(T_i\circ\cdots\circ T_1)x\|
+\|(T_i\circ\cdots\circ T_1) x\|)+\zeta_{i+1,n}\nonumber\\
&+\big\|R_{i+1}\big((W_{i+1}\circ T_{i,n}\circ\cdots T_{1,n})x
+b_{i+1}\big)-R_{i+1}\big((W_{i+1}\circ T_i\circ\cdots\circ T_1)x
+b_{i+1}\big)\big\|
\nonumber\\
&\hskip -4mm\leq (\|W_{i+1}\|+\chi_{i+1,n}) 
\|(T_{i,n}\circ\cdots\circ
T_{1,n}) x-(T_i\circ\cdots\circ T_1) x\|
+\chi_{i+1,n}\|(T_i\circ\cdots\circ T_1)x\|+
\zeta_{i+1,n}\nonumber\\
&\hskip -4mm\leq (\|W_{i+1}\|+\chi_{i+1,n})
(\tau_{i,n}\|x\|+\theta_{i,n})
+\chi_{i+1,n}\Bigg(\|x\|\prod_{k=1}^{i}\|W_k\|+
\sum_{j=1}^{i}\Bigg(\|b_j\|\prod_{k=j+1}^{i}\|W_k\|\Bigg)\Bigg)
+\zeta_{i+1,n}\nonumber\\
&\hskip -4mm=\tau_{i+1,n}\|x\|+\theta_{i+1,n},
\end{align}
which proves the result by induction.
\end{proof}

We can now present the main result of this section on the
asymptotic behavior of Model~\ref{m:2}. The proof of this
result relies on Theorem~\ref{t:1}, which it extends.

\begin{theorem}
\label{t:3}
Consider the setting of Model~\ref{m:2} and let 
$\alpha\in [1/2,1]$. Suppose that
Assumption~\ref{a:main} is satisfied as well as the following:
\begin{enumerate}[label=\rm(\alph*)]
\itemsep0mm 
\item
\label{t:3i}
$F=\Fix T\neq\emp$, where $T=T_m\circ\cdots\circ T_1$.
\item
\label{t:3ii}
$(W_i)_{1\leq i\leq m}$ satisfies Condition~\ref{o:1} with
parameter $\alpha$.
\item
\label{t:3iii}
One of the following is satisfied:
\begin{enumerate}[label=\rm(\roman*)]
\itemsep0mm 
\item
\label{t:3ci}
$\lambda_n\equiv\alpha=1$ and $Tx_n-x_n\to 0$.
\item
\label{t:3cii}
$(\lambda_n)_{n\in\NN}$ lies in $\left]0,1/\alpha\right[$ and
$\sum_{n\in\NN}\lambda_n(1-\alpha\lambda_n)=\pinf$.
\end{enumerate}
\end{enumerate}
Then $(x_n)_{n\in\NN}$ converges weakly to a point 
$\overline{x}_m\in F$ and $(T_1\overline{x}_m,(T_2\circ T_1)
\overline{x}_m,\ldots, (T_{m-1}\circ\cdots\circ T_1)
\overline{x}_m,\overline{x}_m)$
solves \eqref{e:vi1}. Now suppose that, in addition, any of 
the following holds:
\begin{enumerate}
\itemsep0mm 
\setcounter{enumi}{2}
\item
\label{t:3d}
For every $i\in\{1,\ldots,m-1\}$, $R_i$ is weakly sequentially 
continuous.
\item
\label{t:3e}
For every $i\in\{1,\ldots,m-1\}$, $R_i$ is a separable activation
function in the sense of Proposition~\ref{p:4}.
\item
\label{t:3f}
For every $i\in\{1,\ldots,m-1\}$, $\HH_i$ is finite-dimensional.
\item
\label{t:3g}
For some $\varepsilon\in\left]0,1/2\right[$, 
$(\lambda_n)_{n\in\NN}$ lies in  
$[\varepsilon,(1-\varepsilon)(\varepsilon+1/\alpha)]$ and, 
for every $i\in\{1,\ldots,m\}$, $\HH_i=\HH$ and there exists 
$\beta_i\in\zeroun$ such that
$\|W_i-2(1-\beta_i)\Id\|+\|W_i\|\leq 2\beta_i$.
\end{enumerate}
Then, for every $i\in\{1,\ldots,m-1\}$, $(x_{i,n})_{n\in\NN}$ 
converges weakly to 
$\overline{x}_i=(T_i\circ\cdots\circ T_1)\overline{x}_m$ 
and $(\overline{x}_1,\ldots,\overline{x}_m)$ solves \eqref{e:vi1}.
\end{theorem}
\begin{proof}
Let $(y_n)_{n\in\NN}$ be the sequence defined by $y_0=x_0$
and 
\begin{equation}
\label{e:algo5}
\begin{array}{l}
\text{for}\;n=0,1,\ldots\\
\left\lfloor
\begin{array}{ll}
y_{1,n}&\!\!\!=T_1y_n\\
y_{2,n}&\!\!\!=T_2y_{1,n}\\
       &\hskip -1mm\vdots\\
y_{m,n}&\!\!\!=T_my_{m-1,n}\\
y_{n+1}&\!\!\!=y_n+\lambda_n(y_{m,n}-y_n).
\end{array}
\right.\\[2mm]
\end{array}
\end{equation}
For every $n\in\NN$, set $S_n=T_{m,n}\circ\cdots\circ T_{1,n}$. We
derive from \eqref{e:algo4} and \eqref{e:algo5} that
\begin{align}
\label{e:xnxnb1}
(\forall n\in\NN)\quad
\|x_{n+1}-y_{n+1}\| & =
\|x_n+\lambda_n(S_n x_n-x_n)-y_n-\lambda_n(T y_n-y_n)\|\nonumber\\
&\leq\lambda_n\|S_n x_n-T x_n\|+\|x_n-y_n+\lambda_n(T x_n-T
y_n-x_n+y_n)\|.
\end{align}
At the same time, by Proposition~\ref{p:compTinoTi},
there exist $(\tau_{m,n})_{n\in\NN}\in\ell_{+}^{1}$ and
$(\theta_{m,n})_{n\in\NN}\in\ell_{+}^{1}$ such that
\begin{align} 
(\forall n\in\NN)\quad\|S_n x_n-T x_n\|
&\leq\tau_{m,n}\|x_n\|+\theta_{m,n}\label{e:fz72}\\
&\leq\tau_{m,n}(\|x_n-y_n\|+\|y_n\|)+\theta_{m,n}.
\label{e:xnxnb2}
\end{align}
On the other hand, by Theorem~\ref{t:2},
Assumption~\ref{a:main}\ref{a:mainii}, and \ref{t:3ii}, 
$T$ is $\alpha$-averaged.
Hence, there exists a nonexpansive operator $Q\colon\HH\to\HH$ such
that $T=(1-\alpha)\Id+\alpha Q$. Since 
\ref{t:3iii} implies
that $(\lambda_n)_{n\in\NN}$ lies in $\left]0,1/\alpha\right]$, 
we deduce that 
\begin{align}
\label{e:xnxnb3}
(\forall n\in\NN)\quad
\|x_n-y_n+\lambda_n(T x_n-T y_n-x_n+y_n)\|
&=\|(1-\alpha\lambda_n)(x_n-y_n)+ \alpha\lambda_n
(Qx_n-Qy_n)\|\nonumber\\
&\leq (1-\alpha\lambda_n) \|x_n-y_n\|+\alpha\lambda_n
\|Qx_n-Qy_n\|\nonumber\\
&\leq\|x_n-y_n\|.
\end{align}
Altogether \eqref{e:xnxnb1}, \eqref{e:xnxnb2}, and
\eqref{e:xnxnb3} yield
\begin{equation}
(\forall n\in\NN)\quad\|x_{n+1}-y_{n+1}\|\leq
\Big(1+\frac{\tau_{m,n}}{\alpha}\Big)\|x_n-y_n\|+
\frac{1}{\alpha}\big(\tau_{m,n}\|y_n\|+\theta_{m,n}\big).
\end{equation}
However, Theorem~\ref{t:1} guarantees that 
$\delta=\sup_{n\in\NN}\|y_n\|<\pinf$ and therefore that
\begin{equation}
(\forall n\in\NN)\quad
\|x_{n+1}-y_{n+1}\|\leq\Big(1+\frac{\tau_{m,n}}{\alpha}\Big)
\|x_n-y_n\|+\frac{1}{\alpha}\big(\tau_{m,n}\delta+
\theta_{m,n}\big).
\end{equation}
Since $(\tau_{m,n})_{n\in\NN}$ and
$(\tau_{m,n}\delta+
\theta_{m,n})_{n\in\NN}$
are in $\ell_{+}^{1}$, there exists $\nu\in\RP$ such that
$\|x_n-y_n\|\to\nu$ \cite[Lemma~5.31]{Livre1}. Consequently, 
$\delta'=\sup_{n\in\NN}\|x_n\|\leq\delta+
\sup_{n\in\NN}\|x_n-y_n\|<\pinf$.
Now, set
\begin{equation}
(\forall n\in\NN)\quad e_n=\frac{1}{\alpha}(S_n x_n-T x_n).
\end{equation}
Then it follows from \eqref{e:fz72} that
\begin{equation}
\sum_{n\in\NN}\|e_n\|
\leq\frac{1}{\alpha}\sum_{n\in\NN}
\big(\tau_{m,n}\|x_n\|+\theta_{m,n}\big)
\leq\frac{\delta'}{\alpha}\sum_{n\in\NN}\tau_{m,n}+
\frac{1}{\alpha}\sum_{n\in\NN}\theta_{m,n}<\pinf.
\end{equation}
In view of \eqref{e:algo4}, we have
\begin{equation}
\label{e:fz68}
(\forall n\in\NN)\quad x_{n+1}=x_n+\mu_n(Qx_n+e_n-x_n),
\quad\text{where}\quad\mu_n=\alpha\lambda_n\in\rzeroun.
\end{equation}

\ref{t:3ci}: The weak convergence of $(x_n)_{n\in\NN}$ to a point
$\overline{x}_m\in\Fix Q=F$ follows from \eqref{e:fz68} and 
\cite[Theorem~5.33(iv)]{Livre1} by arguing as in
the proof of \cite[Theorem~5.14(i)]{Livre1}.

\ref{t:3cii}: It follows from \eqref{e:fz68} and
\cite[Proposition~5.34(iii)]{Livre1} that $(x_n)_{n\in\NN}$
converges weakly to a point $\overline{x}_m\in\Fix Q=F$.

In \ref{t:3ci}--\ref{t:3cii} above, 
Proposition~\ref{p:10}\ref{p:10iv} ensures that
$(T_1\overline{x}_m,(T_2\circ T_1)\overline{x}_m,\ldots, 
(T_{m-1}\circ\cdots\circ T_1)\overline{x}_m,\overline{x}_m)$ 
solves \eqref{e:vi1}.

\ref{t:3d}--\ref{t:3f}: If one of these assumptions holds, by 
proceeding as in the proof of 
Theorem~\ref{t:1}\ref{t:1d}--\ref{t:1f}, we obtain that,
for every $i\in\{1,\ldots,m-1\}$,
$(T_i\circ\cdots\circ T_1)x_n\weakly\overline{x}_i=
(T_i\circ\cdots\circ T_1)\overline{x}_m$ and that, furthermore,
$(\overline{x}_1,\ldots,\overline{x}_m)$ solves \eqref{e:vi1}.
However, Proposition~\ref{p:compTinoTi} asserts that, for every
$i\in\{1,\ldots,m-1\}$, there exist
$(\tau_{i,n})_{n\in\NN}\in\ell_{+}^{1}$ and
$(\theta_{i,n})_{n\in\NN}\in\ell_{+}^{1}$ such that, for every
$n\in\NN$, 
\begin{equation}
\|x_{i,n}-(T_i\circ\cdots\circ T_1)x_n\| 
=\|(T_{i,n}\circ\cdots\circ T_{1,n})
x_n-(T_i\circ\cdots\circ T_1) x_n\| 
\leq\tau_{i,n}\|x_n\|+\theta_{i,n}.
\end{equation}
Since $(x_n)_{n\in\NN}$ is bounded,
$x_{i,n}-(T_i\circ\cdots\circ T_1)x_n \to 0$
and therefore $x_{i,n}\weakly\overline{x}_i$.

\ref{t:3g}: 
For every $i\in\{1,\ldots,m\}$, set
\begin{equation}
\label{e:defein}
(\forall n\in\NN)\quad
e_{i,n}=(T_{i,n}\circ T_{i-1,n}\circ\cdots\circ T_{1,n}) x_n-
(T_i\circ T_{i-1,n}\circ\cdots\circ T_{1,n}) x_n,
\end{equation}
and let $(\chi_{i,n})_{n\in\NN}$ and $(\zeta_{i,n})_{n\in\NN}$ 
be defined as in \eqref{e:defchizeta}.
By Propositions~\ref{p:boundcompTib}, \ref{p:difTin}, and 
\ref{p:compTinoTi}, we have
\begin{equation}
(\forall n\in\NN)\quad
\|e_{1,n}\|\leq\chi_{1,n}\|x_n\|+\zeta_{1,n}
\end{equation}
and
\begin{align}
&\hskip -14mm
(\forall i\in\{2,\ldots,m\})
(\exi(\tau_{i-1,n})_{n\in\NN}\in\ell_{+}^{1})
(\exi(\theta_{i-1,n})_{n\in\NN}\in\ell_{+}^{1})
(\forall n\in\NN)\quad
\nonumber\\ 
\|e_{i,n}\|&\leq\chi_{i,n}\|(T_{i-1,n}\circ\cdots\circ 
T_{1,n})x_n\|
+\zeta_{i,n}\nonumber\\
&\leq\chi_{i,n}\big(\|(T_{i-1,n}\circ\cdots\circ T_{1,n}) x_n-
(T_{i-1}\circ\cdots\circ T_1)x_n\|
+\|(T_{i-1}\circ\cdots\circ T_1)x_n\|\big)+\zeta_{i,n}\nonumber\\
&\leq\chi_{i,n}\Bigg(\tau_{i-1,n}\|x_n\|+\theta_{i-1,n}
+\|x_n\|\prod_{k=1}^{i-1}\|W_k\|
+\sum_{j=1}^{i-1}\|b_j\|\Bigg(\prod_{k=j+1}^{i-1}
\|W_k\|\Bigg)\Bigg)+\zeta_{i,n}.
\end{align}
Thus, since $(x_n)_{n\in\NN}$ is bounded, 
\begin{equation}
\label{e:kj1979}
(\forall i\in\{1,\ldots,m\})\quad
(\|e_{i,n}\|)_{n\in\NN}\in\ell_+^1.
\end{equation}
In addition, by \eqref{e:defein} and \eqref{e:algo4},
\begin{equation}
(\forall n\in\NN)\quad 
x_{n+1}=x_n+\lambda_n\big(T_m(T_{m-1}(\cdots
T_2(T_1x_n+e_{1,n})+e_{2,n}\cdots)
+e_{m-1,n})+e_{m,n}-x_n\big).
\end{equation}
Thus, since Proposition~\ref{p:9}\ref{p:9iii} and Theorem~\ref{t:2} 
imply that the operators $(T_i)_{1\leq i\leq m}$ are averaged, the
proof can be completed as that of Theorem~\ref{t:1}\ref{t:1g} since 
\cite[Theorem~3.5(ii)]{Jmaa15} asserts that \eqref{e:kj2018} 
remains valid under \eqref{e:kj1979}. 
\end{proof}


\begin{thebibliography}{99}  
\setlength{\itemsep}{-1pt}
\small

\bibitem{Arag18} 
F. J. Arag\'on Artacho and R. Campoy,
A new projection method for finding the closest point in the
intersection of convex sets,
{\em Comput. Optim. Appl.,}
vol. 69, pp. 99--132, 2018.

\bibitem{Atto18}
H. Attouch, J. Peypouquet, and P. Redont,
Backward-forward algorithms for structured monotone inclusions 
in Hilbert spaces,
{\em J. Math. Anal. Appl.,}
vol. 457, pp. 1095--1117, 2018.

\bibitem{Bail78} 
J.-B. Baillon, R. E. Bruck, and S. Reich,
On the asymptotic behavior of nonexpansive mappings and semigroups
in Banach spaces,
{\em Houston J. Math.,}
vol. 4, pp. 1--9, 1978.

\bibitem{Jfan12} 
J.-B. Baillon, P. L. Combettes, and R. Cominetti,
There is no variational characterization of the cycles in the 
method of periodic projections,
{\em J. Funct. Anal.,}
vol. 262, pp. 400--408, 2012.

\bibitem{Barg18} 
C. Bargetz, S. Reich, and R. Zalas,
Convergence properties of dynamic string-averaging projection
methods in the presence of perturbations,
{\em Numer. Algorithms,}
vol. 77, pp. 185--209, 2018.

\bibitem{Barro93}
A. R. Barron,  
Universal approximation bounds for
superpositions of a sigmoidal function,
{\em IEEE Trans. Inform. Theory,}
vol. 39, pp. 930--941, 1993.

\bibitem{Baus96} 
H. H. Bauschke and J. M. Borwein, 
On projection algorithms for solving convex feasibility problems,
{\em SIAM Rev.,}
vol. 38, pp. 367--426, 1996.

\bibitem{Livre1} 
H. H. Bauschke and P. L. Combettes, 
{\em Convex Analysis and Monotone Operator Theory in Hilbert 
Spaces,} 2nd ed. 
Springer, New York, 2017.

\bibitem{Baus15} 
H. H. Bauschke, D. Noll, and H. M. Phan, 
Linear and strong convergence of algorithms involving
averaged nonexpansive operators,
{\em J. Math. Anal. Appl.,}
vol. 421, pp. 1--20, 2015.

\bibitem{Bils00} 
J. Bilski,  
The backpropagation learning with logarithmic transfer function,  
{\em Proc. 5th Conf. Neural Netw. Soft Comput.,}
pp. 71--76, 2000. 

\bibitem{Borw17}
J. M. Borwein, G. Li, and M. K. Tam,
Convergence rate analysis for averaged fixed point iterations in
common fixed point problems,
{\em SIAM J. Optim.,}
vol. 27, pp. 1--33, 2017.

\bibitem{Borw92} 
J. Borwein, S. Reich, and I. Shafrir, 
Krasnoselski-Mann iterations in normed spaces, 
{\em Canad. Math. Bull.,}
vol. 35, pp. 21--28, 1992.

\bibitem{Botr17} 
R. I. Bo\c{t} and E. R. Csetnek, 
A dynamical system associated with the fixed points set of a
nonexpansive operator, 
{\em J. Dynam. Differential Equations,}
vol. 29, pp. 155--168, 2017.

\bibitem{Brav18} 
M. Bravo and R. Cominetti,
Sharp convergence rates for averaged nonexpansive maps,
{\em Israel J. Math.,}
vol. 227, pp. 163--188, 2018.

\bibitem{Brid90} 
J. S. Bridle,
Probabilistic interpretation of feedforward classification network
outputs, with relationships to statistical pattern recognition, In:
{\em Neurocomputing, NATO ASI Series, Series F,}
vol. 68, pp. 227--236. Springer, Berlin, 1990.

\bibitem{Carl17}
B. Carlile, G. Delamarter, P. Kinney, A. Marti, and B. Whitney, 
Improving deep learning by inverse square root linear units 
(ISRLUs), 2017.
\url{https://arxiv.org/abs/1710.09967}

\bibitem{Cegi12} 
A. Cegielski,
{\em Iterative Methods for Fixed Point Problems in Hilbert Spaces,}
Lecture Notes in Mathematics, vol. 2057. Springer, Heidelberg, 2012.

\bibitem{Cens16}
Y. Censor and R. Mansour,
New Douglas--Rachford algorithmic structures and their convergence
analyses,
{\em SIAM J. Optim.,}
vol. 26, pp. 474--487, 2016.

\bibitem{Cras95} 
P. L. Combettes,
Construction d'un point fixe commun \`a une famille de contractions 
fermes,
{\rm C. R. Acad. Sci. Paris S\'er. I Math.,}
vol. 320, pp. 1385--1390, 1995.

\bibitem{Opti04}
P. L. Combettes, 
Solving monotone inclusions via compositions of nonexpansive
averaged operators,
{\em Optimization,} 
vol. 53, pp. 475--504, 2004.

\bibitem{Comb18}
P. L. Combettes,
Monotone operator theory in convex optimization,
{\em Math. Programming,}
vol. B170, pp. 177--206, 2018.

\bibitem{Siop07}
P. L. Combettes and J.-C. Pesquet,
Proximal thresholding algorithm for minimization over 
orthonormal bases,
{\em SIAM J. Optim.,}
vol. 18, pp. 1351--1376, 2007.

\bibitem{Smms05}
P. L. Combettes and V. R. Wajs, 
Signal recovery by proximal forward-backward splitting,
{\em Multiscale Model. Simul.,} 
vol. 4, pp. 1168--1200, 2005.

\bibitem{Jmaa15}
P. L. Combettes and I. Yamada, 
Compositions and convex combinations of averaged nonexpansive
operators,
{\em J. Math. Anal. Appl.,}
vol. 425, pp. 55--70, 2015.

\bibitem{Cond13}
L. Condat, 
A primal-dual splitting method for convex optimization involving
Lipschitzian, proximable and linear composite terms,
{\em J. Optim. Theory Appl.,}
vol. 158, pp. 460--479, 2013. 

\bibitem{Cybe89} 
G. Cybenko,
Approximation by superposition of sigmoidal functions,
{\em Math. Control Signals Systems,}
vol. 2, pp. 303--314, 1989.

\bibitem{Ecks92} 
J. Eckstein and D. P. Bertsekas, 
On the Douglas-Rachford splitting method and the proximal 
point algorithm for maximal monotone operators,
{\em Math. Program.,}
vol. 55, pp. 293--318, 1992.

\bibitem{Elli93}
D. L. Elliot, 
A better activation function for artificial neural networks,
{\em Institute for Systems Research, University of Maryland,}
Tech. Rep. 93-8, 1993.

\bibitem{Funa89} 
K.-I. Funahashi, 
On the approximate realization of continuous mappings by neural 
networks, 
{\em Neural Netw.,}
vol. 2, pp. 183--192, 1989. 

\bibitem{Glor11}
X. Glorot, A. Bordes, and Y. Bengio,
Deep sparse rectifier neural networks,
{\em Proc. 14th Int. Conf. Artificial Intell. Stat.,}
pp. 315--323, 2011.

\bibitem{Hayk98} 
S. Haykin, 
{\em Neural Networks: A Comprehensive Foundation,}
2nd ed. Pearson Education, Singapore, 1998.

\bibitem{Heka12}
K. He, X. Zhang, S. Ren, and J. Sun,
Delving deep into rectifiers: Surpassing human-level performance on
imagenet classification,
{\em Proc. Int. Conf. Comput. Vision,}
pp. 1026--1034, 2015.

\bibitem{He16}
K. He, X. Zhang, S. Ren, and J. Sun,
Deep residual learning for image recognition,
{\em Proc. IEEE Conf. Comput. Vision Pattern Recogn.,}
pp. 770-778, 2016.

\bibitem{Lecu15}
Y. A. LeCun, Y. Bengio, and G. Hinton,
Deep learning,
{\em Nature,}
vol. 521, pp. 436--444, 2015.

\bibitem{Lecu98}
Y. A. LeCun, L. Bottou, G. B. Orr, and K.-R. M\"uller,
Efficient backprop,
{\em Lect. Notes Comput. Sci.,}
vol. 1524, pp. 9--50, 1998.

\bibitem{Mart72} 
B. Martinet, 
D\'etermination approch\'ee d'un point fixe d'une application 
pseudo-contractante. Cas de l'application prox,
{\em C. R. Acad. Sci. Paris,}
vol. A274, pp. 163--165, 1972.

\bibitem{Mccu43} 
W. S. McCulloch and W. H. Pitts, 
A logical calculus of the ideas immanent in nervous activity,
{\em Bull. Math. Biophys.,}
vol. 5, pp. 115--133, 1943.

\bibitem{Mour18} 
W. M. Moursi,
The forward-backward algorithm and the normal problem,
{\em J. Optim. Theory Appl.,}
vol. 176, pp. 605--624, 2018.

\bibitem{Nair10}
V. Nair and G. E. Hinton,
Rectified linear units improve restricted Boltzmann machines, 
\emph{Proc. 27st Int. Conf. Machine Learn.,} 
pp. 807--814, 2010.

\bibitem{Rock70} 
R. T. Rockafellar, 
{\em Convex Analysis.}
Princeton University Press, Princeton, NJ, 1970.

\bibitem{Roc76a} 
R. T. Rockafellar, 
Monotone operators and the proximal point algorithm,
{\em SIAM J. Control Optim.,}
vol. 14, pp. 877--898, 1976.

\bibitem{Rose58}
F. Rosenblatt,
The perceptron: A probabilistic model for information storage and
organization in the brain, 
{\em Psychological Rev.,}
vol. 65, pp. 386--408, 1958.

\bibitem{Svriv15}
R. K. Srivastava, K. Greff, and J. Schmidhuber,
Training very deep networks,
{\em Proc. Neural Inform. Process. Syst. Conf.,}
vol. 28, pp. 2377--2385, 2015.

\bibitem{Tariy16}
S. Tariyal, A. Majumdar, R. Singh, and M. Vatsa,
Deep dictionary learning,
{\em IEEE Access}, 
vol. 4, pp. 10096--10109, 2016.

\bibitem{Tsen92} 
P. Tseng, 
On the convergence of products of firmly nonexpansive mappings,
{\em SIAM J. Optim.,}
vol. 2, pp. 425--434, 1992.

\bibitem{Yama17} 
M. Yamagishi and I. Yamada,
Nonexpansiveness of a linearized augmented Lagrangian operator for
hierarchical convex optimization,
{\em Inverse Problems,}
vol. 33, art. 044003, 35 pp., 2017.

\bibitem{Zhan01}
X.-P. Zhang,
Thresholding neural network for adaptive noise reduction, 
{\em IEEE Trans. Neural Netw.,} 
vol. 12, pp. 567--584, 2001.





\end{thebibliography}
\end{document}